\documentclass{amsart} 
\usepackage{amscd,amssymb,amsmath,amsbsy,amsthm}
\usepackage{color}
\usepackage[all]{xy}
\usepackage[colorlinks,plainpages,backref,urlcolor=blue]{hyperref}
\usepackage[width=5.64in,height=7.96in,centering]{geometry}
\usepackage{fancyvrb,bbm}
\usepackage[mathscr]{euscript}
\usepackage{tikz}
\usetikzlibrary{calc,intersections,arrows}

\newtheorem{theorem}{Theorem}[section]
\newtheorem{lemma}[theorem]{Lemma}
\newtheorem{corollary}[theorem]{Corollary}
\newtheorem{proposition}[theorem]{Proposition}

\theoremstyle{definition}
\newtheorem{definition}[theorem]{Definition}
\newtheorem{exm}[theorem]{Example}
\newtheorem{rem}[theorem]{Remark}

\newenvironment{example}%
{\pushQED{\qed}\begin{exm}}%
{\popQED\end{exm}}  

\newenvironment{remark}%
{\pushQED{\qed}\begin{rem}}%
{\popQED\end{rem}}  

\newcommand{\abs}[1]{\left|#1\right|}
\newcommand{\set}[1]{\left\{#1\right\}}

\newcommand{\Z}{\mathbb{Z}}
\newcommand{\PP}{\mathbb{P}}
\newcommand{\Q}{\mathbb{Q}}
\newcommand{\R}{\mathbb{R}}
\newcommand{\C}{\mathbb{C}}

\renewcommand{\k}{\Bbbk}  
\newcommand{\ksheaf}{\tilde{\k}}
\newcommand{\one}{\mathbbm{1}}

\newcommand{\Cover}{\mathscr{C}}  
\newcommand{\UU}{\mathscr{U}}  

\newcommand{\sV}{\mathsf{V}}

\newcommand{\m}{\mathfrak{m}}

\DeclareMathOperator{\Tot}{Tot}
\DeclareMathOperator{\Ext}{Ext}
\DeclareMathOperator{\ShExt}{\mathscr{E}\kern -0.5pt xt}
\DeclareMathOperator{\Hom}{Hom}
\DeclareMathOperator{\ShHom}{\mathscr{H}\kern -0.5pt om}
\DeclareMathOperator{\id}{id} 
\DeclareMathOperator{\ab}{ab}  

\DeclareMathOperator{\gr}{gr}
\DeclareMathOperator{\lk}{lk}   
\DeclareMathOperator{\st}{st}   
\DeclareMathOperator{\rank}{rank}
\DeclareMathOperator{\res}{res}  
\DeclareMathOperator{\opp}{op}  

\DeclareMathOperator{\corank}{corank}

\DeclareMathOperator{\depth}{depth}
\DeclareMathOperator{\MCM}{MCM} 
\DeclareMathOperator{\dep}{dp}  
\DeclareMathOperator{\gd}{gd}  

\newcommand{\Fbar}{\overline{F}} 

\newcommand{\inj}{\hookrightarrow}

\newcommand{\A}{{\mathcal{A}}}

\newcommand{\F}{{\mathcal{F}}}
\newcommand{\G}{{\mathcal{G}}}
\newcommand{\CC}{{\mathcal{C}}} 
\newcommand{\I}{{\mathcal{I}}} 

\newcommand{\HH}{{\mathcal H}} 
\newcommand{\ZZ}{{\mathcal Z}}
\newcommand{\N}{{\mathcal N}}  
\newcommand{\M}{{\mathcal{A}}}  

\newcommand{\Fr}{{\left.\F\right|}}  

\def\dot{\mathchar"013A}  
\newcommand{\hdot}{{\raise1pt\hbox to0.35em{\huge $\dot$\!}}}

\newenvironment{romenum}
{ 

\begin{enumerate}}{\end{enumerate}}

\title[Combinatorial covers and vanishing of cohomology]%
{Combinatorial covers and vanishing of cohomology}

\author[G. Denham]{Graham Denham$^1$} 
\address{Department of Mathematics, University of Western Ontario,
London, ON  N6A 5B7}
\email{\href{mailto:gdenham@uwo.ca}{gdenham@uwo.ca}}
\urladdr{\href{http://www.math.uwo.ca/~gdenham}%
{http://www.math.uwo.ca/\~{}gdenham}}
\thanks{$^1$Partially supported by 
NSERC (Canada)}

\author[A.~I. Suciu]{Alexander~I.~Suciu$^2$}
\address{Department of Mathematics,
Northeastern University,
Boston, MA 02115}
\email{\href{mailto:a.suciu@neu.edu}{a.suciu@neu.edu}}
\urladdr{\href{http://www.northeastern.edu/suciu/}%
{http://www.northeastern.edu/suciu/}}
\thanks{$^2$Partially supported by 
NSF grant DMS--1010298 and NSA grant H98230-13-1-0225}

\author[S. Yuzvinsky]{Sergey~Yuzvinsky}
\address{Department of Mathematics, 
University of Oregon, Eugene, OR 97403} 
\email{\href{mailto:yuz@uoregon.edu}{yuz@uoregon.edu}}
\urladdr{\href{http://pages.uoregon.edu/yuz/}%
{http://pages.uoregon.edu/yuz/}}

\begin{document}

\begin{abstract}  
We use a Mayer--Vietoris-like spectral sequence to establish 
vanishing results for the cohomology of complements of 
linear and elliptic hyperplane arrangements, as part of a 
more general framework involving duality and abelian duality 
properties of spaces and groups.  In the process, we consider 
cohomology of local systems with a general, Cohen--Macaulay-type 
condition.  As a result, we recover known vanishing 
theorems for rank-$1$ local systems as well as group ring coefficients, 
and obtain new generalizations.  
\end{abstract}



\maketitle
\tableofcontents

\section{Introduction}
\label{sect:intro}

\subsection{Overview}
\label{subsec:intro1}
A remarkable feature of the theory of cohomology jump loci, 
first brought to light by Eisenbud, Popescu and Yuzvinsky in \cite{EPY03}, 
is that the resonance varieties of hyperplane arrangements 
``propagate."  This raises the natural question: What is the 
topological underpinning of this phenomenon?  
Inspired by recent work of Davis, Januszkiewicz, 
Leary, and Okun \cite{DJLO11}, we show that complements 
of hyperplane arrangements, as well as many other 
spaces have cohomology concentrated in a single degree, for local
systems satisfying suitable hypotheses.  

Our cohomological vanishing results will be applied in \cite{DSY-duality} 
to study the relationship between the duality properties of a 
space (and its universal abelian cover), and the propagation 
properties of its characteristic and resonance varieties.
We will also apply our machinery in \cite{DSY-novikov} to the 
computation of Novikov homology for arrangement complements, 
deducing a new bound on the $\Sigma$-invariants for arrangements.

In order to establish such vanishing theorems, we construct here a
Mayer--Vietoris-type spectral sequence, which is related to 
a spectral sequence developed by Davis and Okun in \cite{DO12} in a 
related context.  Our version refines the usual underlying filtration, in order 
to control the $E_2$ page, while requiring weaker hypotheses 
than those of Davis and Okun.  We obtain vanishing results for
local systems satisfying Cohen--Macaulay-type conditions on spaces
which, in a loose sense, also possess Cohen--Macaulay properties.

\subsection{Combinatorial covers and spectral sequences}
\label{subsec:intro2}
The input for the spectral sequence is a device which we call a combinatorial
cover: that is, a countable cover $\Cover$ of a space $X$, which is 
either open, or closed and locally finite, together with an 
order-preserving, surjective map $\phi\colon N(\Cover)\to P$ 
from the nerve of the cover to a ranked poset $P$ (with rank function 
$\rho$), for which certain compatibility conditions are satisfied.
Given such a cover and a locally constant sheaf $\F$ on $X$, 
we have a spectral sequence starting at 
\begin{equation}
\label{eq:ss}
E_2^{pq}=\prod_{x\in P} 
H^{p-\rho(x)}\big({\phi^{-1}(P_{\leq x})}, {\phi^{-1}(P_{<x})};\,
H^{q+\rho(x)}(X,\Fr_{U_x})\big),
\end{equation}
converging to $H^{p+q}(X,\F)$, where $U_x=\bigcap_{U\in S} U$, 
for some 
set $S\in N(\Cover)$ with $\phi(S)=x$, and
$P_{<x}$, $P_{\leq x}$ are (strict) lower intervals in the poset $P$.

Such covers can be constructed for spaces with a combinatorial stratification:
for example, our approach is particularly well-suited when studying 
an arrangement $\A$ of submanifolds in a smooth manifold $Y$.
We assume that each submanifold in $\A$ is either compact or open; 
moreover, any two submanifolds intersect transversely, and each intersection 
has only finitely many connected components.  We then show 
that the complement $M(\A)$ admits a combinatorial cover 
indexed by the stratification poset.  Moreover, the $E_2$ page 
from \eqref{eq:ss} specializes to
\begin{equation}
\label{eq:ss-bis}
E_2^{pq}=\prod_{X\in P}
H^{p-\dim(X)}\big(X,D_X;\,H^{q+\dim(X)}(M(\A),\Fr_{U_X})\big), 
\end{equation}
where $D_X$ is the union of all strata properly contained in $X$.

\subsection{Hyperplane arrangements}
\label{subsec:intro4}
Now suppose $\A$ is an arrangement of complex hyperplanes in 
$V=\C^n$. We let $L(\A)$ be its intersection poset, ordered opposite 
to inclusion, and $L_{\geq1}(\A)=L(\A)\setminus\{V\}$.  For each flat 
$X\in L_{\geq1}(\A)$, we define an element 
$\gamma_X\in \pi_1(M(\A))$ which, upon restriction to the 
subarrangement $\A_X$ of hyperplanes containing $X$, lies  
in the center of the group $\pi_1(M(\A_X))$.  The homology 
classes $a_X=[\gamma_X]$ live in the free abelian group 
$H_1(M(\A),\Z)$ and satisfy the compatibility condition 
$a_X=\sum_{H\in\A\colon H\leq X} a_H$.

We employ the de~Concini--Procesi wonderful compactification to 
construct a particularly convenient combinatorial cover of $M(\A)$.
Using the spectral sequence \eqref{eq:ss-bis} associated to this cover, 
we generalize well-known vanishing results for cohomology 
with coefficients in rank $1$ local systems on arrangement 
complements, \cite{Ko86, ESV92, STV95}.

Previous vanishing results have restricted attention to local systems 
given by finite-dimen\-sional group representations with field coefficients.
In order to understand a vanishing result for (integer) group ring coefficients
due to Davis, Januskiewicz, Leary and Okun \cite{DJLO11} within the same
framework, however, we note that the restrictions of finite-dimensionality 
and field coefficients can be avoided.

The classical vanishing result can be formulated as follows:
suppose $\k$ is a field and $A$ is a finite-dimensional 
$\k[\pi_1(U(\A))]$-module where $U(\A)$ is the complement of the compactification  of $\A$. 
 If the invariant submodule 
$A^{\gamma_X}=0$ for all flats
$X\in L_{\geq1}(\A))$
satisfying the Euler characteristic $\chi(U(\A_X))\neq0$, then
\begin{equation}
\label{eq:vanish}
 \text{$H^p(U(\A),A)=0$ for all $p\neq n-1$}.
\end{equation}

More generally, we show that $\pi_1(U(\A))$ contains free abelian 
subgroups $C_S$ of rank $n$ generated by certain subsets $S$ 
of the elements $\set{\gamma_X\colon X\in L_{\geq1}(\A)}$, 
to wit, the nested sets. We regard $A$ as a module over the 
(commutative) algebras $\k[C_S]$ by restriction, and formulate 
a necessary condition for the vanishing result \eqref{eq:vanish}
to hold in terms of regular sequences.  The classical version is 
obtained as a special case.

\subsection{Elliptic arrangements}
\label{subsec:intro5}
Our machinery also applies very well to elliptic arrangements. 
Let $E$ be an elliptic curve.  An elliptic arrangement in $E^{\times n}$
is a finite collection of fibers of group homomorphisms 
$E^{\times n}\to E$, see \cite{LV12, Bi13}.  We start our 
analysis by showing that the complement of such an 
arrangement $\A$ is a Stein manifold, provided $\A$ 
is essential. 

Using our spectral sequence, we obtain a vanishing result 
analogous to \eqref{eq:vanish}.  As an application, we recover a recent 
theorem of Levin and Varchenko \cite{LV12} on vanishing of cohomology 
for what they term `convenient' rank $1$ local systems. 
We also infer that the complement of an elliptic arrangement $\A$ 
is both a duality and an abelian duality space of dimension $n+r$, 
where $r$ is the corank of $\A$.  

A special case of this construction is the configuration 
space of $n$ points on an elliptic curve $E$, which  
is a classifying space for the pure elliptic braid group  $PE_n$. 
Our approach shows that the group $G=PE_n$ is both a duality 
and an abelian duality group of dimension $n$.  Roughly speaking, 
this means the cohomology of $G$ with coefficients in $\Z{G}$ 
and $\Z{G^{\ab}}$, respectively, is concentrated in degree $n$; 
for details and further references, we refer to \cite{DSY-duality}.

\subsection{Toric complexes}
\label{subsec:intro6}
Another context where the techniques developed here apply very well 
is that of the moment-angle complexes corresponding to the pair $(S^1,*)$.  
Given a finite, $d$-dimensional simplicial complex $L$ on vertex set $[n]$, 
the corresponding {\em toric complex}, $T_L$, is a subcomplex of the 
$n$-torus, whose $k$-cells are in one-to-one correspondence 
with the $(k-1)$-simplices of $L$.  Notably, the fundamental group 
$G_L=\pi_1(T_L)$ is the right-angled Artin group associated to the 
$1$-skeleton of $L$.  

A natural condition that comes into play in this setting is 
that of Cohen--Macaulayness.  The simplicial complex $L$ is said  
to be a {\em Cohen--Macaulay}\/ complex if for each simplex $\sigma\in L$
with the link $\lk(\sigma)$, 
the reduced cohomology $\widetilde{H}^*(\lk(\sigma),\Z)$ 
is concentrated in degree $\dim L -\abs{\sigma}$ and is torsion-free. 
As shown in \cite{BM01, JM05}, 
the group $G_L$, is a duality group if and only if the 
corresponding flag complex, $\Delta_L$, is Cohen--Macaulay.  

As a counterpoint, we show that the toric complex $T_L$ is an 
abelian duality space if and only if the simplicial complex $L$ itself is 
Cohen--Macaulay.  Moreover, if $L$ is Cohen--Macaulay over a field $\k$, 
and $A$ is a maximal Cohen--Macaulay-module over $\k[G_\tau]$, 
for each $\tau\in L$, we show that the cohomology groups 
$H^p(T_L,A)$ vanish, for all $p\neq d+1$.

\section{Combinatorial covers and spectral sequences}
\label{sect:ss}

\subsection{Combinatorial covers}
\label{subsec:comb cov}

Let $X$ be a paracompact space, and let $\Cover$ be a cover 
of $X$.  We can then form the {\em nerve}\/ of the cover, $N(\Cover)$:  
this is the poset (ordered by inclusion) of all finite collections 
$S=\{U_1,\ldots,U_p\}$ of elements of $\Cover$ for which 
$\cap S:=\bigcap_{i=1}^p U_i$ is non-empty. 

Given a poset $P$, we let $\abs{P}$ denote its {\em order complex}: 
this is the simplicial complex with simplices corresponding 
to chains in $P$.  We say $P$ is a ranked poset if it admits a rank 
function $\rho\colon P\to \Z$ compatible with the order and consistent 
with the covering relation of the order.

\begin{definition}
\label{def:wcc}
A {\em combinatorial cover}\/  for $X$ is a pair
$(\Cover,\phi)$, where 
\begin{enumerate}
\item \label{wc1}
$\Cover$ is a countable cover which is either open, or closed 
and locally finite.  
\item  \label{wc2}
$\phi\colon N(\Cover)\to P$ is an order-preserving, surjective map 
from the nerve of the cover to a ranked
poset $P$, with the property that if $S\le T$ and 
$\phi(S)=\phi(T)$, then the inclusion $\cap T \hookrightarrow \cap S$ 
admits a homotopy inverse.  
\item \label{wc2'}
If $S\leq T$ and $\bigcap S=\bigcap T$, then 
$\phi(S)=\phi(T)$.
\end{enumerate} 
\end{definition}

\begin{definition}
\label{def:cover}
We will say that $(\Cover,\phi)$ is a {\em strong combinatorial cover}\/ 
of $X$ if it is a combinatorial cover and, moreover, the map 
$\phi$ induces a homotopy equivalence of simplicial complexes, 
$\phi\colon \abs{N(\Cover)}\to \abs{P}$.
\end{definition}

\begin{remark}
\label{rem:trivial}
Any cover  $\Cover$ satisfying condition \eqref{wc1} can be turned 
into a (strong) combinatorial cover in a trivial fashion:   simply take 
$P=N(\Cover)$ and $\phi$ the identity map of this poset.   Typically,
however, $P$ is chosen to reflect some underlying combinatorial structure
which is independent of the choice of cover.  The condition that 
$\phi$ induces a homotopy equivalence is particularly convenient, but 
too restrictive for applications like those in \S\ref{sect:hyparr-coho}.
\end{remark}

\begin{example}
\label{ex:toy}
Let $X$ be a space homeomorphic to a $2$-disk with four punctures, covered
by $3$ open sets as shown in Figure~\ref{fig:comb_cover}.
\begin{figure}
\[
\def\dx{0.5}
\def\dy{0.577}
\Cover:\hspace*{5pt}
\begin{tikzpicture}[scale=0.5,fill opacity=0.5,draw opacity=1,
             fill=cyan!40!white,baseline=(current bounding box.center)]
\filldraw { (0,-\dy) [rounded corners] -- (2*\dx,2*\dy) -- (-2*\dx,2*\dy)}
      -- (0,-\dy);
\filldraw[rotate=-120] { (0,-\dy) [rounded corners] -- (2*\dx,2*\dy) -- (-2*\dx,2*\dy)}   -- (0,-\dy);
\filldraw[rotate=120] { (0,-\dy) [rounded corners] -- (2*\dx,2*\dy) -- (-2*\dx,2*\dy)}   -- (0,-\dy);
\filldraw[fill=white,opacity=1] (-0.6*\dx,1.2*\dy) circle (4pt);
\filldraw[fill=white,opacity=1] (0.6*\dx,1.2*\dy) circle (4pt);
\draw[fill=white,opacity=1,rotate=-120] (0,1.2*\dy) circle (4pt);
\draw[fill=white,opacity=1,rotate=120] (0,1.2*\dy) circle (4pt);
\node[opacity=1] at (0,3.5*\dy) {$U_1$};
\node[opacity=1] at (3.5*\dx,-2*\dy) {$U_2$};
\node[opacity=1] at (-3.2*\dx,-2*\dy) {$U_3$};
\end{tikzpicture}
\hspace*{14pt}
N(\Cover):
\hspace*{5pt}
\begin{tikzpicture}[baseline=(current bounding box.center)]
\node (123) at (0,3) {$\set{U_1,U_2,U_3}$};
\node (12) at (-1.5,2) {$\set{U_1,U_2}$};
\node (13) at (0,2) {$\set{U_1,U_3}$};
\node (23) at (1.5,2) {$\set{U_2,U_3}$};
\node (1) at (-1.5,1) {$\set{U_1}$};
\node (2) at (0,1) {$\set{U_2}$};
\node (3) at (1.5,1) {$\set{U_3}$};
\draw (1) -- (12) -- (2) -- (23) -- (3) -- (13) -- (1);
\draw (12) -- (123) -- (13);
\draw (23) -- (123);
\end{tikzpicture}
\hspace*{14pt}
P: \hspace*{5pt}
\begin{tikzpicture}[baseline=(current bounding box.center)]
\node (123) at (0,1) {$*$};
\node (1) at (-0.75,0) {$1$};
\node (2) at (0,0) {$2$};
\node (3) at (0.75,0) {$3$};
\draw (1) -- (123) -- (2);
\draw (123) -- (3);
\end{tikzpicture}
\]
\caption{A strong combinatorial cover}\label{fig:comb_cover}
\end{figure}
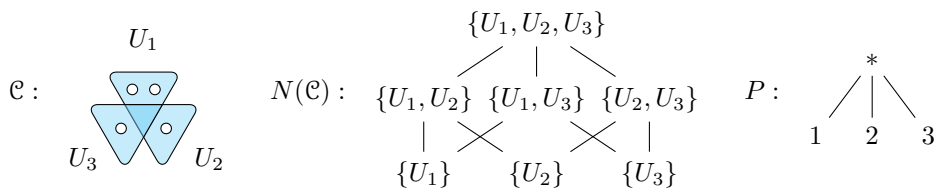

Define $\phi\colon N(\Cover)\to P$ by letting $\phi(\set{U_i})=i$ 
for $1\leq i\leq 3$, and $\phi(S)=*$ if $\abs{S}\neq 1$.  
To see $(\Cover,\phi)$ is a combinatorial 
cover, note that the intersections of any two or more open sets are 
all identical, so that $\cap S=\cap T$ for any $S,T\in \phi^{-1}(*)$, 
verifying the only non-vacuous instances of condition \eqref{wc2}.  
In fact, this is a strong combinatorial cover, since the map 
$\phi\colon \abs{N(\Cover)}\to \abs{P}$ is a homotopy 
equivalence of contractible complexes.
\end{example}

\subsection{Sheaf cohomology}
\label{subsec:sheaves}
Let $\k$ be a (commutative) coefficient ring, $X$ a path-connected space,
and $A$ a left $\k[\pi_1(X)]$-module.  We regard $A$ as a local system
on $X$.  We are interested in the cohomology $H^{\hdot}(X,A)$,
and we recall that this is the same as the cohomology $H^{\hdot}(X,\F)$
of a locally constant sheaf $\F$ of $\k$-modules determined by $A$.  
In particular, the stalks $\F_x\cong A$ as $\k$-modules for every 
$x\in X$.  Our results apply to locally constant sheaves, then, 
although (more generally) constructible sheaves appear in some
arguments.  We note that we do not assume that our 
stalks are finitely-generated over $\k$, following \cite{Sch03},
although the book \cite{Di04} remains a good reference.

Given a space $X$ with an open cover $\Cover$, there is a tautological
map $\pi\colon X\to N(\Cover)$, given by $\pi(x)=\set{i\colon x\in U_i}$.
The order topology on $N(\Cover)$ makes $\pi$ continuous.  If $\F$ is a 
locally constant sheaf on $X$, the Leray spectral sequence of $\pi$ 
unifies \v{C}ech cohomology and the Mayer--Vietoris principle:
\begin{equation}
\label{eq:e2pq}
E_2^{pq}=H^p(N(\Cover),R^q\pi_*\F)\Rightarrow H^{p+q}(X,\F).
\end{equation}

For $S\in N(\Cover)$, let $[S,\infty)$ denote the principal order ideal, a 
principal open subset.  By definition, 
\begin{align}
R^q\pi_*\F([S,\infty))&= H^q(U_S,i^*\F)\\
&= H^q(X,\Fr_{U_S}),\notag
\end{align}
where $U_S=\pi^{-1}([S,\infty))$.  Write $\HH^q=R^q\pi_*\F$ for short.

For the cohomology of a sheaf in the order topology, we refer to \cite{deH62}.
In particular, since the poset $N(\Cover)$ is the face poset of a simplicial
complex, if $\G$ is any sheaf, then $H^{\hdot}(N(\Cover),\G)$ 
can be computed from
\begin{equation}
\label{eq:simplicial}
C^p(N(\Cover),\G):= \prod_{S\in N_p(\Cover)}\G(S),
\end{equation}
with a simplicial differential~\cite[\S 11.4]{deH62}.

\begin{remark}
\label{rem:pairs}
Following the notation of \cite[IV.8.1]{Iv}, if $Z$ is closed in $X$ and
$j\colon X\setminus Z\hookrightarrow X$ denote the embedding of 
the complement, then for $p\geq 0$ let
\begin{equation}
\label{eq:sheaf_pair}
H^p(X,Z;\F):= H^p(X,j_!j^*\F).
\end{equation}
In particular, if both $X$ and $X\setminus \{*\}$ are 
path-connected, and $\F$ is a local system on $X$, then
the notation
\begin{equation}
\widetilde{H}^p(X,\F):= H^p(X,*;\F)
\end{equation}
is well-defined.
\end{remark}

\begin{remark}
\label{rem:supports}
If $X$ is compact, then we can also interpret \eqref{eq:sheaf_pair} as 
compactly supported cohomology, since then
\begin{align}
\label{eq:shriek}
H^p(X,j_!j^*\F) &= H^p_c(X,j_!j^*\F)\\
&= H^p_c(X\setminus Z,j^*\F),  \notag
\end{align}
following \cite[III.7.3]{Iv}.
\end{remark}

\subsection{Poincar\'e-Verdier duality}
\label{subsec:verdier}
For the following classical result, we refer to \cite[Theorem~3.3.1]{Di04} 
in the finite-dimensional case.  

\begin{theorem}
\label{thm:PV}
Suppose $X$ is a $d$-dimensional complex manifold, and $\F$ is a locally
constant sheaf of $\k$-modules on $X$.  If $\k$ is a field, 
for $0\leq p\leq 2d$, we have
\begin{equation*}
\label{eq:verdier}
H_c^p(X,\F)^\vee\cong H^{2d-p}(X,\F^\vee),
\end{equation*}
where $-^\vee=\Hom(-,\k)$.  More generally, suppose $\k$ is a 
principal ideal domain and
the stalks of $\F$ are free $\k$-modules.  Then, for all $p$, 
there is a ``universal coefficients'' short exact sequence,
\[
\xymatrix@C-0.8em{
0\ar[r] & \Ext^1_\k(H^{p+1}_c(X,\F),\k)\ar[r] &
H^{2d-p}(X,\F^\vee)\ar[r] & H^p_c(X,\F)^\vee\ar[r] & 0.
}
\]
\end{theorem}

\begin{proof}
Let $\ksheaf$ denote the constant sheaf on $X$.  Since $X$ is
canonically oriented,
we identify its orientation sheaf with $\ksheaf$.
Then, in the bounded derived category of sheaves on $X$, Verdier
duality \cite[\S3.3.1]{KSbook} simplifies to the isomorphism
\begin{equation}
\label{eq:rgamma}
R\Gamma(X,R\ShHom(\F,\ksheaf[2d]))\cong
R\Hom_\k(R\Gamma_c(X,\F),\k).
\end{equation}
The left side gives rise to a Grothendieck spectral sequence with
\begin{align}
\label{eq:groth}
E_2^{pq}&=H^{p+2d}(X,\ShExt^q(\F,\ksheaf))\\
\notag
&= \begin{cases}
H^{p+2d}(X,\F^\vee)&\text{if $q=0$}\\
0 &\text{otherwise,}
\end{cases}
\end{align}
using \cite[Lemma, p.~126]{Bo84} to reduce the calculation to stalks,
together with our assumption that the stalks are free.  So this
spectral sequence degenerates at $E_2$.  Computing
hyper-Ext on the right-hand side gives another spectral sequence with
\begin{equation}
\label{eq:hyperext}
E_2^{pq}=\Ext^p_\k(H^{-q}_c(X,\F),\k)
\end{equation}
with the same abutment.  Then $E_2^{pq}=0$ unless $p=0$ or $1$, and the 
short exact sequences are obtained by comparing the two spectral sequences:
this one gives short exact sequences with $E_\infty$ in the middle, and
the other one identifies $E_\infty=H^\cdot(X,\F^\vee)$.
\end{proof}

Duality gives the following useful trick, which goes back to \cite{Ko86}.  
Recall that a complex manifold is said to be a {\em Stein manifold}\/ if it 
can be realized as a closed, complex submanifold of some affine space $\C^k$.  
As is well-known, every Stein manifold of (complex) dimension $n$ 
has the homotopy type of a CW-complex of dimension $n$.

\begin{corollary}
\label{cor:cpct vanishing trick}
Suppose that $\F$ is a locally constant sheaf of $\k$-modules on a Stein
manifold $X$ of dimension $n$, where $\k$ is a principal ideal domain.
If the stalks of $\F$ are free $\k$-modules, then 
$H_c^p(X,\F)=0$ for $0\leq p<n$.
\end{corollary}

\begin{proof}
Cohomology with coefficients in local systems is an invariant 
of homotopy type, so the Stein property
implies that $H^{n-p}(X,\F^\vee)=0$ for $0\leq p<n$.
By Theorem~\ref{thm:PV}, 
\begin{equation}
\label{eq:exthpc}
\Ext^i_\k(H^p_c(X,\F),\k)=0
\end{equation}
for $0\leq p<n$ and all $i$.  It follows that $H_c^p(X,\F)=0$ as well.
\end{proof}

\subsection{A spectral sequence}
\label{subsec:spectral sequence}

Now suppose $\F$ is a locally constant sheaf of $\k$-modules 
on $X$, where $\k$ is either $\Z$, or a field. Suppose also that 
$X$ admits a combinatorial cover $(\Cover,\phi)$.   
We use these data to construct a spectral sequence converging 
to $H^\hdot(X,\F)$.

Put an (arbitrary) total order on the elements of $\Cover$.  
Then we obtain an exact (\v{C}ech) complex of sheaves, 
$0\to\F\to\CC^{\hdot}$, where for $p\geq0$,
\begin{equation}
\label{eq:ccp}
\CC^p=\CC^p(\Cover)=\prod_{U_1<\cdots<U_{p+1}}
\Fr_{U_1\cap \cdots\cap U_{p+1}}.
\end{equation}

Here each set $U_i$ belongs to the cover $\Cover$, and the complex 
$\CC^{\hdot}$ has the \v{C}ech differential induced by restriction maps: 
see, for example, \cite[Th.~5.2.1]{God}.  More formally, 
for a subspace $U\stackrel{i}{\hookrightarrow} X$, by $\Fr_U$ we mean
the sheaf $i_*i^*\F$ on $X$, and we note that 
(1) since $\F$ is locally constant,
the higher direct image functors of $i$ vanish, and (2) we also have
$\Fr_U=\bar{i}_!\bar{i}^*\F$,
where $\bar{i}$ denotes the inclusion of the closure of $U$ in $X$.

For each integer $q\ge 0$, our hypotheses 
imply that the subposet $(\rho\phi)^{-1}(q)$ is a disjoint union 
of incomparable components indexed by $\rho^{-1}(q)$, 
and for $S$ within a single component, the homotopy 
type of $\cap S$ is constant.  We note that any order-preserving
map $\rho\colon P\to \Z$ could be used in place of the rank function,
provided that the fibers of $\rho$ are antichains.

For each $x\in P$, let $P_{\leq x}=\set{y\in P\mid y\le x}$. 
Then $\phi^{-1}(P_{\leq x})$ is a sub-poset of $N(\Cover)$.  
Choosing a set $S\in N(\Cover)$ with $\phi(S)=x$, 
write $U_x=\cap S$; then $U_x$ is well-defined up to 
homotopy.

Given a combinatorial cover $(\Cover,\phi)$, consider
the restriction of $\HH^q$ to a fiber $\phi^{-1}(x)
\subseteq N(\Cover)$.  By condition \eqref{wc2}, the restriction is locally
constant on the connected components of the fiber.
For each $x\in P$, let 
$j_x\colon \phi^{-1}(x)\hookrightarrow N(\Cover)$ denote the 
inclusion of the fiber.  

\begin{theorem}
\label{th:weak cohomology ss}
If $X$ has a combinatorial cover $(\Cover,\phi)$, 
then for every  locally constant sheaf $\F$ on $X$ there is 
a spectral sequence with
\[
E_2^{pq}=\prod_{x\in P} 
H^{p-\rho(x)}\big({\phi^{-1}(P_{\leq x})}, {\phi^{-1}(P_{<x})};\,
H^{q+\rho(x)}(X,\Fr_{U_x})\big),
\]
converging to $H^{p+q}(X,\F)$,
where $\rho$ denotes the rank function on $P$.
\end{theorem}

\begin{proof}
For each $p\geq0$, the sheaf $\CC^p$ is a product indexed 
by $N_p(\Cover)$.  Thus, we may form a Cartan--Eilenberg resolution 
$\I^{\hdot\hdot}$ with $\I^{p\hdot}=\prod_{S\in N_p(\Cover)}\I_S^{p\hdot}$, 
so that $\I^{p\hdot}_S$ is an injective resolution of 
$\Fr_{\cap S}$ for each $S\in N_p(\Cover)$. 

We filter the global sections $\Gamma(\I^{\hdot\hdot})$ of the double complex
by rows.  The spectral sequence obtained in this way 
collapses to $H^p(X,H^0(\CC^{\hdot}))=
H^p(X,\F)$, by the acyclicity of $\CC^{\hdot}$.  
We may also filter the double complex by setting, for $p\in\Z$,
\begin{equation}
\label{eq:fptot}
F^p(\Tot\Gamma(\I^{\hdot\hdot})) = 
\prod_{S\in N_s(\Cover)\colon s+\rho(\phi(S))\geq p} 
\Gamma(\I_S^{\hdot\hdot}).
\end{equation}
Then the $E_0$ term of the filtration spectral sequence is given by
\begin{align}
 \notag
E_0^{pq}&=\gr^p\Tot\Gamma(\I^{\hdot\hdot})^{p+q},\\
&=\prod_{\substack{S\in N_s(\Cover)\colon s+t=p+q,\\
s+\rho(\phi(S))=p}}\Gamma(\I^{s,t}).
\end{align}
Hence,
\begin{align}
E_1^{pq}&=
\prod_{\substack{S\in N_s(\Cover)\colon s+t=p+q,\\
s+\rho(\phi(S))=p}} H^t(X,\Fr_{\cap S}), \notag\\
&=
\prod_{S\in N_{p-\rho(\phi(S))}(\Cover)} 
H^{q+\rho(\phi(S))}(X,\Fr_{\cap S}). \label{eq:E1term} 
\end{align}

For a fixed $x\in P$, let $(E_1^{pq})_x$ denote the factors of \eqref{eq:E1term}
indexed by $S$ for which $\phi(S)=x$.
By construction, 
\begin{align}
(E^{pq}_1)_x&=\prod_{\substack{
S\in\phi^{-1}(x)\colon \\
\notag
\abs{S}-1+\rho(x)=p}} H^{q+\rho(x)}(X,\Fr_{U_x})\\
&=\prod_{\substack{
S\in\phi^{-1}(P_{\leq x})\colon \\
\abs{S}-1+\rho(x)=p}} (j_x)_!j_x^*\HH^{q+\rho(x)}
\end{align}

The differential $d_1\colon E_1^{pq}\to E_1^{p+1,q}$ is induced 
by the \v{C}ech differential in $\CC^\cdot$.
Suppose that $S\in N_p(\Cover)$, $T\in N_{p+1}(\Cover)$, and
$S\subseteq T$.  By hypothesis, if $\phi(S)=\phi(T)$, then 
the inclusion $\cap T \hookrightarrow \cap S$ is a homotopy equivalence; 
hence, the restriction $\Fr_S\to \Fr_T$ induces an isomorphism in cohomology.  
That is, the cohomology sheaf $\HH^{q+\rho(x)}$ is locally constant on the
open set $\phi^{-1}(x)$.  
On the other hand, if $\phi(S)\neq\phi(T)$, then $\rho(\phi(S))<\rho(\phi(T))$,
by our assumption that fibers of $\rho$ are antichains.  It follows
in this case that the restriction is zero on $E_1^{pq}$.
Now, using \eqref{eq:simplicial}, we have
\begin{equation}
\label{eq:epq1}
(E^{pq}_1)_x= C^{q+\rho(x)}(N(\Cover),(j_x)_!j_x^*\HH^{q+\rho(x)}),
\end{equation}
where the differential on the right agrees with the restriction 
of $d_1$.  The claim follows.
\end{proof}
\begin{corollary}
\label{cor:Davis_trick}
Suppose $X$ has a strong 
combinatorial cover $(\Cover,\phi)$, and $\F$ is a 
locally constant sheaf on $X$. There is then a spectral sequence with
\[
E^{pq}_2=\prod_{x\in P}
\widetilde{H}^{p-\rho(x)-1}(\lk_{\abs{P}}(x);\,H^{q+\rho(x)}(X,\Fr_{U_x})), 
\]
converging to $H^{p+q}(X,\F)$. 
\end{corollary}

\begin{proof}
We consider the map of order complexes 
$\abs{\phi}\colon \abs{N(\Cover)}\to\abs{P}$ induced by $\phi$.  
For any sheaf $\G$ on $N(\Cover)$, we have
$\phi_*(j_x)_!j_x^*\G=(i_x)_!i_x^*\phi_*\G$, 
where $i_x\colon \abs{P_{\leq x}}\hookrightarrow\abs{P}$.

Since $(j_x)_!\G$ is locally constant on each fiber of $\phi$ and
the fibers are, by hypothesis, contractible, the Leray
spectral sequence of $\phi$ degenerates, so
\begin{align}
H^{p-\rho(x)}(N(\Cover),(j_x)_!j_x^*\HH^{q+\rho(x)})&\cong
H^{p-\rho(x)}(\abs{P}, (i_x)_!i_x^*\phi_*\HH^{q+\rho(x)})\\
&\cong  \notag 
\widetilde{H}^{p-\rho(x)-1}(\lk_{\abs{P}}(x);\,H^{q+\rho(x)}(X,\Fr_{U_x}) ),
\end{align}
where the second isomorphism is obtained by using 
the long exact sequence of the inclusion $i_x$, and rewriting 
the restriction of the sheaf $\phi_*\HH^{q+\rho(x)}$ as a local system
on $\lk_{P}(x)$.
\end{proof}

\begin{remark}
\label{rem:davis okun}
For the sake of the readers' intuition, we explain the role of the rank
function $\rho$ in the filtration.  One obtains the classical 
Mayer--Vietoris spectral sequence by filtering the complex $\I^{\hdot\hdot}$
by the cardinality $s$ of elements of the nerve.  Up to reindexing, 
the $E_1$ page of that spectral sequence is the same as 
our $E_1$ page from \eqref{eq:E1term}.  However, in general,
the classical differential $d_1$ depends both on the combinatorics of the
cover as well as the local system, so a closed-form description of the 
$E_2$ page is not available in general.  

Davis and Okun solve this problem in \cite{DO12} in a closely related setting
by imposing hypotheses (``(Z)'' and ``(Z')'') that are sufficiently restrictive
to permit a description of the $E_2$ page analogous to that of Theorem~%
\ref{th:weak cohomology ss}.  

Our formalism of combinatorial covers and filtration \eqref{eq:fptot}
is chosen to give a spectral sequence with explicit $E_2$ page without
having to impose hypotheses on the local system.  Our filtration is 
finer than the classical one, in such a way that this $d_1$ depends only
on the combinatorics of the cover.
\end{remark}

\section{Submanifold arrangements}
\label{sect:cc mfd}

In \cite{DJLO11}, Davis, Januszkiewicz, Leary and Okun use the notion 
of a strong 
combinatorial cover (which they call a ``small cover'') for hyperplane 
arrangement complements. Their construction restricts an 
open cover of the ambient affine space to the complement 
(see Figure \ref{fig:cover} for a simple example).  In this 
section, we show that this idea works in greater generality.

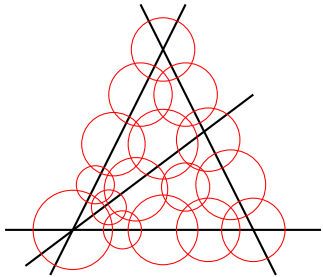
\begin{figure}
\centering
\begin{tikzpicture}[scale=0.6]
\draw[style=thick] (-0.5,3) -- (2.5,-3);
\draw[style=thick]  (0.5,3) -- (-2.5,-3);
\draw[style=thick] (-3.5,-2) -- (3.5,-2);
\draw[style=thick]  (-3.05,-2.8) -- (2,1);
\draw[color=red] (-2,-2) circle (25pt);
\draw[color=red] (0,2) circle (20pt);
\draw[color=red] (1,0) circle (20pt);
\draw[color=red] (2,-2) circle (20pt);
\draw[color=red] (0.5,1) circle (20pt);
\draw[color=red] (1.5,-1) circle (22pt);
\draw[color=red] (1,-2) circle (20pt);
\draw[color=red] (0,-2) circle (22pt);
\draw[color=red] (-0.9,-2) circle (12pt);
\draw[color=red] (-1.5,-1) circle (12pt);
\draw[color=red] (-1.2,-1.5) circle (11pt);
\draw[color=red] (-0.6,-1.1) circle (20pt);
\draw[color=red] (-0.5,1.0) circle (20pt);
\draw[color=red] (-1.1,-0.1) circle (20pt);
\draw[color=red] (0,-0.1) circle (22pt);
\draw[color=red] (0.5,-1.05) circle (15.2pt);
\end{tikzpicture}
\caption{A strong combinatorial cover for a line arrangement}
\label{fig:cover}
\end{figure}

\subsection{Open covers for submanifold arrangements}
\label{subsec:sub mfd}
To begin, let $Y$ be a smooth, connected manifold, and let 
$\A=\set{W_1,\ldots, W_n}$ be a finite collection of proper, 
connected submanifolds of $Y$.  We will assume 
that the intersection of any subset of $\A$ 
is also smooth, and has only finitely many connected components, 
and call such a collection an {\em arrangement of submanifolds}. 

Let $L(\A)$ denote the collection of all connected components of
intersections of zero or more submanifolds.  Then $L(\A)$ forms a 
finite poset under inclusion.  For every submanifold $X\in L(\A)$, let 
\begin{equation}
\label{eq:asubx}
\A_X=\set{W\in\A\colon X\subseteq W}
\end{equation}
be the {\em closed subarrangement}\/ associated to $X$, and let 
\begin{equation}
\label{eq:asupx}
\A^{X} = 
\set{W\cap X \mid W\in \A\setminus \A_X}
\end{equation} 
be the {\em restriction}\/ of $\A$ to $X$.  We say $\A$ is {\em unimodular}\/ 
if each submanifold $X\in L(\A)$ is connected.

Recall that an open cover of a $k$-dimensional manifold is ``good'' if
all sets in the cover and all their nonempty, finite intersections are 
diffeomorphic to $\R^k$. Such covers always exist, and are cofinal 
in the set of all covers, cf.~ \cite[Corollary 5.2]{BT}.  Compact manifolds, 
as well as open manifolds (i.e., interiors of compact manifolds with 
boundary) admit {\em finite}\/ good covers.  

Now fix a Riemannian metric on $Y$, and let $d$ be the associated 
distance function.  For any subset $X\subset Y$ and $\epsilon>0$, let
\begin{equation}
\label{eq:xepsilon}
X_\epsilon=\set{y\in Y\colon d(x,y)<\epsilon~\text{for some $x\in X$}}.
\end{equation}

\begin{theorem}
\label{thm:pre_combinatorial}
Let $\M$ be an arrangement of submanifolds in $Y$.  
There exists a finite open cover $\UU$ of $Y$ for which
\begin{enumerate}
\item \label{enu:pre_comb1} 
$\UU$ is closed under intersections, and 
\item \label{enu:pre_comb2}
the poset
\begin{equation*}
L_U:=\set{X\in L(\M)\colon X\cap \overline{U}\neq\emptyset}
\end{equation*}
has a unique minimum, for each $U\in{\mathscr U}$.
\end{enumerate}

If, additionally, each intersection of  submanifolds in $\M$ 
has a finite good cover, then $Y$ has a finite, good open 
cover $\UU$ with the properties above.
\end{theorem}

\begin{proof}
We construct such a cover by induction as follows.  
Suppose there exist a pair $(L,\UU)$ where $L$ is 
an order ideal of $L(\M)$, and $\UU$ is a finite collection of open 
sets in $Y$ that covers each element $X\in L$, and which satisfies 
conditions \eqref{enu:pre_comb1} and 
\eqref{enu:pre_comb2}, above.  If $L\neq L(\M)$, we show that there 
exists another pair $(L',\UU')$ with the same property, for which $L'$
strictly contains $L$.  Since the intersection poset is finite, this
implies the claim.

In the notation above, suppose $(L,\UU)$ is such a pair, and that
$X$ is a minimal element of $L(\M)\setminus L$.  We let $L'=L\cup\set{X}$ and
show we can extend the covering $\UU$ to $X$.  By hypothesis, the 
function $m\colon \UU\to L(\M)$ defined by $m(U)=\min L_U$ is 
well-defined.  

We let 
\begin{align}
\label{eq:xcirc}
X^\circ &=X\setminus \bigcup_{U\in \UU}U\\
&= \bigcup_{U\in \UU\colon X\in L_U}U, \notag
\end{align}
and claim that there exists a value $\epsilon>0$ for which 
\begin{romenum}
\item \label{enu:pf1}
For each $U\in\UU$, we have 
$X^\circ_\epsilon\cap \overline{U}\neq \emptyset
\Rightarrow X\cap \overline{U}\neq \emptyset$; and
\item \label{enu:pf2}
For each $X'\in L(\M)$, we have $X'\cap \overline{X^\circ_\epsilon}
\neq \emptyset \Leftrightarrow X\leq X'$.
\end{romenum}
That is, $X^\circ_\epsilon$ is a neighborhood of $X$, minus the 
open sets covering the lower-dimensional strata.  We claim $X^\circ_\epsilon$ 
can be chosen to intersect only (closures of) open sets that already intersect 
$X$, and to intersect only those strata that contain $X$.

To argue that such a neighborhood exists, we note that $X$ is closed and
$\UU$ is finite, so the number
\begin{equation}
\label{eq:eps1}
\epsilon_1:=\min_{U\colon X\cap\overline{U}=\emptyset}
\inf\set{ d(x,y)\colon x\in X, y\in \overline{U}}
\end{equation}
is strictly positive. Hence, condition  \eqref{enu:pf1} is satisfied 
whenever $\epsilon<\epsilon_1$.

Note that the implication ``$\Leftarrow$'' in condition \eqref{enu:pf2} is
satisfied trivially for all $X'$ and $\epsilon$. Also note that 
$X'\cap X^\circ\neq\emptyset$ implies $X'\geq X$.
Thus, for the other implication we need to find $\epsilon>0$ 
such that, for all $X'\in L(\M)$, we have
$X'\cap \overline{X_\epsilon^\circ}\neq\emptyset\Rightarrow X'\cap 
X^\circ\neq\emptyset$.

Since each $X'$ is closed and $L(\M)$ is finite, the number 
\begin{equation}
\label{eq:eps2}
\epsilon_2:=\min_{X'\colon X'\cap X^\circ=\emptyset}
\inf\set{d(x,y)\colon x\in X', y\in X^\circ}
\end{equation}
is strictly positive.  Therefore, condition  \eqref{enu:pf2} is satisfied if 
$\epsilon<\epsilon_2$.

Let $V=X^\circ_\epsilon$, for brevity, and put 
\begin{equation}
\label{eq:uu}
\UU'=\UU\cup\set{V}\cup\set{V\cap
U\colon U\in\UU\text{~and~}V\cap U\neq\emptyset}.
\end{equation}

Clearly, $\UU'$ is again a finite open cover of $X$.  
By induction, $\UU'$ is closed under intersections, Property
\eqref{enu:pre_comb1} of the cover.
It remains to check Property \eqref{enu:pre_comb2}, that
the posets $L_{V}$ and $L_{V\cap U}$ have unique minimal elements.

In the former case, condition \eqref{enu:pf2} directly implies
$X$ is the minimal
element in $L_{V}$.  In the latter case, if $X'\cap V\cap U\neq \emptyset$,
then $X\leq X'$ by \eqref{enu:pf2}.  On the other hand, by
\eqref{enu:pf1}, we have $X\in L_{V\cap U}$, so again $X$ is the
(unique) minimal element.

Finally, we assume that each stratum $X$ admits a finite good cover, 
in which case $X^\circ$ also has a finite good cover.  Then $X^\circ_\epsilon$ 
can be replaced by a finite union of open sets $\set{U}$
for which each such $U$ is diffeomorphic to some $\R^k$ and 
satisfies $m(U)=X$, and similarly with all the intersections.
\end{proof}

\subsection{A combinatorial cover over the intersection poset}
\label{subsec:cc sub mfd}
As before, let $\A$ be an arrangement 
of submanifolds in a manifold $Y$. We now restrict the open cover from
Theorem~\ref{thm:pre_combinatorial} to the {\em complement}\/ 
of the arrangement, 
\begin{equation}
\label{eq:complement}
M(\A)=Y\setminus \bigcup_{W\in \M} W, 
\end{equation} 
to obtain a combinatorial cover.

\begin{theorem}
\label{thm:cc}
Let $\M$ be an arrangement of submanifolds in a manifold $Y$, and 
let $M$ be its complement.
\begin{enumerate}
\item If each submanifold $W\in \M$ is either compact or open, 
then $M$ has a combinatorial cover $(\Cover,\phi)$ over 
the intersection poset $L(\M)$.
\item If, moreover, each subspace $X\in L(\M)$ is contractible, 
then $(\Cover,\phi)$ is a strong combinatorial cover. 
\end{enumerate}
\end{theorem}

\begin{proof}
Let $\UU$ be the finite open cover of $Y$ given by 
Theorem~\ref{thm:pre_combinatorial}, and let
\begin{equation}
\label{eq:cover2}
\Cover=\set{U\cap M\colon U\in{\mathscr U}}, 
\end{equation}
an open cover of $M$. Define a map $g\colon \Cover\to L(\M)$ by 
$g(U\cap M)=m(U)$.  Clearly, $g$ is an order-reversing map.  

The unique minimum property \eqref{enu:pre_comb2} 
of $\UU$ implies that a collection 
of sets $S=\set{U_1,\ldots,U_k}$ in $\Cover$ has a nonempty 
intersection if and only if $\set{g(U_1),\ldots, g(U_k)}$ 
form a chain in some order.  Then $g$ induces a poset map
$\phi\colon N(\Cover)\to L(\M)$ by letting 
\begin{equation}
\label{eq:phis-bis}
\phi(S)=\max\nolimits_{L(\M)}\set{g(U)\colon U\in S}.
\end{equation}

We may define a map $\rho\colon L(\M)\to\Z$ by letting 
$\rho(X)=\dim X$: then, according to Definition~\ref{def:cover},
the pair $(\Cover,\phi)$ is indeed a combinatorial cover of $M$.

Finally, if each subspace $X\in L(\M)$ is contractible,
then, by Quillen's Lemma, the above map $\phi$ is a 
homotopy equivalence.  Thus, the triple 
$(\Cover,\phi)$ is a strong combinatorial cover.
\end{proof}

For the rest of this section we will assume that any two submanifolds 
in our arrangement $\M$ intersect transversely.  Under this assumption, 
every subspace $X$ in the intersection poset of $\M$ is itself a smooth 
(not necessarily connected) submanifold of $Y$. Write 
\begin{equation}
\label{eq:dx}
D_X=\bigcup_{Z\in L(\M) : {Z<X}} Z.
\end{equation}

\begin{corollary}
\label{cor:ss mdf arr}
Let $\M$ be an arrangement of submanifolds of a manifold $Y$, such that 
each $W\in \M$ is either compact or open, and any two submanifolds 
intersect transversely.  Let $M$ be the complement of the arrangement, 
let $(\Cover,\phi)$ be the combinatorial cover from Theorem \ref{thm:cc}, 
and let $\F$ be a locally constant sheaf on $M$.  
There is then a spectral sequence with
\begin{equation*}
\label{eq:ss arr}
E_2^{pq}=\prod_{X\in L(\M)} H^{p-\rho(X)}(X,D_X;
H^{q+\rho(X)}(M,\F_{U_X})),
\end{equation*}
converging to $H^{p+q}(M,\F)$. Here 
$U_X$ denotes an open set in $\Cover$ for which $m(U_X)=X$.  
\end{corollary}  

\begin{proof}
We note that $N(\Cover)=N(\UU)$.  So the
fiber $\phi^{-1}(L(\M)_{\leq X})$ can be identified with 
$N(\UU_X)$, the nerve of the subcover of $\UU$ 
that intersects $X$.  Since $\UU$ is a good cover, by the Nerve Lemma, 
$N(\UU_X)$ is homotopic to $X$, and we apply the 
the spectral sequence from Theorem~\ref{th:weak cohomology ss}.
\end{proof}

\begin{remark}
\label{rem:xdx}
If all submanifolds $W\in A$ are compact, then, by Remark \ref{rem:supports}, 
we can replace in the above formula the cohomology of the pair $(X,D_X)$ by 
the compactly supported cohomology of the complement, $X\setminus D_X$. 
\end{remark}
\subsection{A combinatorial cover of the arrangement}
\label{subsec:cover-arr}
The previous section constructs a combinatorial cover of the complement
of an arrangement whose open sets are local complements.  Here, let us
cover the arrangement itself, so that the elements of the cover are 
homotopic to the intersections of submanifolds themselves.  Let us
further assume that intersections of submanifolds consist of finite
unions of equidimensional components.

Given a submanifold arrangement, let $L'(\A)$ denote the poset of 
intersections of one or more submanifolds.  Unlike $L(\A)$, intersection
makes $L'(\A)$ a meet-semilattice; however, the elements of $L'(\A)$ need
not be connected.  For $X$ in $L'(\A)$, let $\rho(X)=\dim X - k$, where
$k$ is the largest number of maximal submanifolds in $\A$ containing $X$:
then if $X\subsetneq Y$, $\rho(X)<\rho(Y)$.  Let
\begin{align}
\label{eq:sigma a}
\Sigma(\A)&=\bigcup_{W\in\A}W\\ \notag
&=\bigcup_{X\in L'(\A)}X.
\end{align}

For each $X\in L'(\A)$, let $U_\epsilon(X)$ be an open $\epsilon$-neighborhood
of the manifold $X$ in $Y$.  By our local finiteness hypothesis, if $\epsilon>0$
is small enough, then $U_\epsilon(X)$ deformation-retracts to $X$, and
$U_\epsilon(X\cap Y)=U_\epsilon(X)\cap U_\epsilon(Y)$ whenever $X\cap Y\neq
\emptyset$.  

Let $\Cover=\set{U_\epsilon(X)\colon X\in L'(\A)}$, and define a map 
$\phi\colon 
N(\Cover)\to L'(\A)^{\opp}$ as follows: 
if $S:=(U_\epsilon(X_1),\ldots,U_\epsilon(X_k))\in 
N(\Cover)$, let $\phi(S)=X_1\cap\cdots \cap X_k\in L'(\A)^{\opp}$.

\begin{theorem}
\label{thm:cover-arr}
Let $\A$ be an arrangement of submanifolds in a manifold $Y$.  Then
$(\Cover,\phi)$ is a strong combinatorial cover of $\Sigma(\A)$.
\end{theorem}

\begin{proof}
The claim that $(\Cover,\phi)$ is a combinatorial cover (Definition~\ref{def:cover})
follows by construction.  To see that $\phi$ induces a homotopy equivalence
of order complexes, for $X\in L'(\A)^{\opp}$, note that 
$L'(\A)^{\opp}_{\leq X}$ 
consists of all intersections of submanifolds that contain $X$.  Then 
$U_\epsilon(X)$
is a cone point in the subcomplex $\phi^{-1}(L'(\A)^{\opp}_{\leq X})$, 
so the result follows by Quillen's Lemma~\cite{Qu78}.
\end{proof}

This theorem leads to a spectral sequence complementary to the one of 
Corollary~\ref{cor:ss mdf arr}:

\begin{corollary}
\label{cor:ss cover-arr}
Let $\A$ be an arrangement of submanifolds in a manifold $Y$, and let 
$(\Cover,\phi)$ be the combinatorial cover from Theorem~\ref{thm:cover-arr}.  
For any locally constant sheaf $\F$ on $\Sigma(\A)$, there is a spectral sequence
with
\[
E_2^{pq}=\prod_{X\in L'(\A)^{\opp}} \tilde{H}^{p+\rho(X)-1}(
\lk_{L'(\A)^{\opp}}(X); H^{q-\rho(X)}(X,\F_X)
\]
converging to $H^{p+q}(\Sigma(\A),\F)$.
\end{corollary}

\begin{proof}
We note that $L'(\A)^{\opp}$ is ranked by $-\rho$, and apply 
Corollary~\ref{cor:Davis_trick}.
\end{proof}

\section{Hyperplane arrangements}
\label{sect:hyparr-coho}

\subsection{Intersection lattice and complement}
\label{subsec:lama}
Let $\A$ be an arrangement of hyperplanes in a 
finite-dimensional complex vector space $V=\C^n$. 
We will assume throughout that $\A$ is \emph{central}, that is, 
all the hyperplanes pass through the origin $0\in \C^n$. 

The \emph{intersection lattice}, $L(\A)$, is the ranked 
poset of all intersections of hyperplanes in $\A$, 
ordered by reverse inclusion, and ranked by codimension.  
This is a geometric lattice, with join given by taking the 
intersection of two flats, and meet given by taking the 
linear span of their sum.  For $0\leq p\leq n$, let $L_p(\A)$ 
denote the set of elements of $L(\A)$ of rank $p$. 

The complement of the arrangement, $M(\A)=\C^n\setminus \bigcup_{H\in\A}H$,  
is a connected, smooth complex quasi-projective variety.  Moreover,  
$M(\A)$ is a Stein manifold, and thus it has the homotopy type of a 
CW-complex of dimension at most $n$.  In fact, $M(\A)$ splits 
off the linear subspace  $\C^d = \bigcap_{H\in \A} H$, and so  
it is homotopic to a cell complex of dimension $n-d$. 

The arrangement $\A$ defines by projectivization 
an arrangement $\bar{\A}$ of codimension $1$ 
projective subspaces in $\PP^{n-1}$, with complement 
$U(\A)=\PP^{n-1} \setminus \bigcup_{H \in \A} \bar{H}$. 
The Hopf fibration, 
$\pi\colon \C^n\setminus \{0\} \to \PP^{n-1}$, 
restricts to a smooth, principal fibration, 
$\pi\colon M(\A) \to U(\A)$, with fiber $\C^*$.  
Fixing a hyperplane $H_0\in \A$ yields a trivialization of 
the bundle map $\C^n\setminus H_0 \to \PP^{n-1} \setminus \bar{H}_0$, 
and thus produces a diffeomorphism $M(\A)\cong \C^*\times U(\A)$. 

Let us fix an ordering of the hyperplanes, 
$\A=\{H_1,\dots ,H_r\}$, and choose linear forms 
$f_i\colon \C^{n}\to \C$ with $\ker(f_i)=H_i$.  Assembling 
these forms together, we obtain a linear map
$f\colon V\to\C^{r}$, sending $z\mapsto (f_1(z), \dots ,f_r(z))$. 
Without much loss of generality, we will assume from now on that 
$\A$ is \emph{essential}, i.e., the intersection of all the hyperplanes in $\A$ 
is precisely $0$, or, equivalently, the map $f$ is injective. 

Clearly, $f$ restricts to an inclusion $f\colon M(\A)\inj (\C^*)^{r}$, 
equivariant with respect to the diagonal action of $\C^*$ 
on both source and target.  Thus, $f$ descends to a map 
$\overline{f}\colon M(\A)/\C^*\inj (\C^*)^{r}/\C^*$, which 
defines an embedding $\overline{f}\colon U(\A) \inj(\C^*)^{r-1}$. 

\subsection{Flats and meridians}
\label{subsec:h1 gens}

Fix a basepoint $x_V\in M(\A)$ sufficiently close to $0$. 
We proceed to describe certain generating sets for the fundamental 
group of the complement, $G=\pi_1(M(\A),x_V)$, and for the 
``local" fundamental groups, $G_X=\pi_1(M(\A_X),x_V)$, 
where recall $\A_X=\set{H\in\A\colon H\leq X}$, for 
each flat $X\in L(\A)$. 

\begin{lemma}
\label{lem:retract}
Let $X$ be a flat in $L(\A)$, and let $i_X\colon M(\A)\to M(\A_X)$ 
be the natural inclusion.  There is then a basepoint-preserving map 
$r_X\colon M(\A_X)\to M(\A)$ so that $i_X \circ r_X\simeq \id$ (rel $x_V$). 
\end{lemma}

\begin{proof}
Choose a point $x_X$ in the relative interior of $X$ in such a way that
the real interval $[x_V,x_X]$ intersects only hyperplanes of $\A_X$.  Let
$B_X$ be the Minkowski sum of $[x_V,x_X]$ with a closed ball of radius 
$\epsilon$, chosen to be small enough so that again 
$B_X\cap H\neq\emptyset$ if and only if $H\leq X$.  

Let $q_X\colon M(\A_X)\to B_X\cap M(\A)$ denote the map
which is the identity on $B_X\cap M(\A_X)$ and radial projection 
from $x$ to $\partial B_X$ on $M(\A_X)\setminus B_X$.  This is 
easily seen to be a strong deformation retraction that, moreover, 
preserves the basepoint (see Figure \ref{fig:retract}).  
Composing now the map $q_X$ with the inclusion 
$B_X\cap M(\A) \inj M(\A)$ yields the desired map $r_X$.
\end{proof}

Let $j\colon \C^* \to M(\A)$ be the inclusion of the fiber 
of the Hopf fibration containing the basepoint $x_V$,  
and let $j_{\sharp}\colon \pi_1(\C^*,1)\to \pi_1(M(A),x_V)$ 
be the induced homomorphism. Upon identifying $\pi_1(\C^*)=\Z$, 
we may define an element $\gamma_{0} \in G$ by 
\begin{equation}
\label{eq:gamma0}
\gamma_0=j_{\sharp}(1).
\end{equation}
By construction, $\gamma_0$ belongs to the center of $G$.  

Similarly, for each flat $X\in L(\A)$, we have a central element 
$\gamma_0^X \in G_X$.  Define then an element 
$\gamma_X\in G$ by 
\begin{equation}
\label{eq:gammax}
\gamma_X = (r_X)_{\sharp}(\gamma_0^X).
\end{equation}

In particular, for each hyperplane $H\in \A$, we have a ``meridian" 
loop $\gamma_H$, based at $x_V$.  A standard application of the 
van Kampen theorem shows that the set $\{\gamma_H \mid H\in \A\}$ 
generates the fundamental group $G=\pi_1(M(\A),x_V)$.  
We also let $a_X=[\gamma_X]$ 
be the corresponding homology classes in $G^{\ab}=H_1(M(\A),\Z)$.  

As noted in \cite{DS14}, the homomorphism induced by the 
map $f\colon M(\A)\to (\C^*)^r$ on fundamental groups 
may be identified with the abelianization map, $\ab\colon G\to G^{\ab}$. 
Furthermore, the group $G^{\ab}$ is the free abelian group on 
$\set{a_H\colon H\in \A}$.

Finally, let $\bar{x}_V\in U(\A)$ be the image of the basepoint $x_V\in M(\A)$ 
under the Hopf map $\pi\colon M(\A)\to U(\A)$, and, for each flat $X$, let 
$\bar{\gamma}_X=\pi\circ \gamma_X$ be the corresponding 
loop.  The fundamental group $\bar{G}=\pi_1(U(\A) , \bar{x}_V)$ is 
the quotient of $G$ by the cyclic subgroup $\Z=\langle \gamma_0\rangle$. 
Moreover, the elements $\{\bar\gamma_H \mid H\in \A\}$ generate 
$\bar{G}$, while the homology classes $\bar{a}_H=[\bar\gamma_H ]$ 
freely generate the group $\bar{G}^{\ab}=H_1(U(\A),\Z)$, subject to the single 
relation $\sum_{H\in \A} \bar{a}_H=0$.

\subsection{Compatibility of the retraction maps}
\label{subsec:compatible}
For each pair of flats $X\leq Y$ in $L(\A)$, let $i_{X,Y}\colon M(\A_Y)\to
M(\A_X)$ denote the inclusion.  Recovering our old notation, $i_{X,0}=i_X$.
Using the construction of Lemma~\ref{lem:retract}, we let
\begin{equation}
\label{eq:defrxy}
r_{X,Y}=i_Y\circ r_X\colon M(\A_X)\to M(\A_Y).
\end{equation}

Applying Lemma~\ref{lem:retract} to the arrangement $\A_Y$, we conclude
that $i_{X,Y}\circ r_{X,Y}\simeq \id$, via a basepoint-preserving homotopy.
We will need two lemmas.

\begin{lemma}
\label{lem:rfunctorial}
For any triple of flats $X,Y,Z\in L(\A)$, if $X\leq Y$
and $Y\leq Z$, then 
\begin{enumerate}
\item \label{rf1} $i_{X,Z}= i_{X,Y}\circ i_{Y,Z}$. 
\item \label{rf2} $r_{X,Z}\simeq r_{Y,Z}\circ r_{X,Y}$.
\end{enumerate}
In particular, taking $Z=0$, we get $r_X\simeq r_Y\circ r_{X,Y}$.
\end{lemma}

\begin{proof}
The first claim follows from the definition.  To show the second, let
$P_X$ denote the image of the retract $r_X$, for each $X\in L(\A)$.  In
the notation of Lemma~\ref{lem:retract}, we have $P_X=M(\A_X)\cap B_X$.
It is enough to show, for any $X\leq Y$, that $r_Y\circ i_Y\circ r_X\simeq
r_X$.  By Lemma~\ref{lem:retract}, the corestriction $r_X\colon M(\A_X)\to
P_X$ is a homotopy equivalence, so it is enough to show that 
$r_Y\circ i_Y\colon P_X\to M(\A)$ is a homotopy equivalence onto its 
image.

\begin{figure}
\def\epsilon{0.1}
\def\delta{0.12}
\begin{tikzpicture}[vertex/.style={circle,draw,inner sep=1pt,fill=black},
scale=1.2,>=latex']
\clip (1,1) circle (2);
\draw (-1,0) -- (4,0);
\draw (0,-1) -- (0,4);
\draw (-1,-2) -- (2,4);  
\node[vertex] (xX) at (1,2) [label={[label distance=5pt]left:$x_X$}] {};
\node[vertex] (x0) at (0,0) [label={[label distance=6pt]above left:$x_Y$}] {};
\node[vertex] (xV) at (1.8,1) [label={[label distance=9pt]right:$x_V$}] {};
\coordinate (xi) at ($(xX)!0.4!(xV)$);
\coordinate (mVX) at ($(xX)!1.13!(xV)$);
\coordinate (startX) at ($(mVX)+(\epsilon,0.8*\epsilon)$);
\coordinate (startY) at ($(x0)!1.12!(xV)+(\delta,-1.8*\delta)$);
\draw[style=dashed] (x0) -- (xV) -- (xX);
\draw[rounded corners=5pt, name path=PX] (startX) -- ++($1.26*(-0.8,1)$) 
-- ++(-2*\epsilon, -2*0.8*\epsilon)
 -- ++($1.26*(0.8,-1)$) -- cycle;
\draw[rounded corners=8pt, name path=PY] (startY) -- ++($1.23*(-1.8,-1)$) 
-- ++(-2*\delta, 2*1.8*\delta)  -- ++($1.23*(1.8,1)$) -- cycle;
\node[vertex] (x) at ($(x0)!0.33!(xi)$) 
[label={[label distance=2pt]above left:$x$}] {};
\draw[->,style=dotted, name path=Line] ($(x0)!1.4!(xi)$) -- (x);
\draw[style=very thick,color=red] ($(x0)!0.94!(xi)$) -- ($(x0)!1.05!(xi)$);
\node at (1.9,1.7) {$P_X$};
\node at (1.4,0.2) {$P_Y$};
\end{tikzpicture}
\caption{The fiber of radial projection}
\label{fig:retract}
\end{figure}

For this, we consider a fiber $(r_Y\circ i_Y)^{-1}(x)$ over a point $x$ in
the image.  If $x\in P_Y^\circ$, then the fiber consists of just the point
$x$.  Otherwise, $x$ is in the image of the radial projection towards
$x_Y$, so the fiber is the intersection of $P_X$ with a ray.  
Since $x\not\in Y\subseteq X$, the intersection of the ray with $X$ is
the point $x_Y$.  
It follows that the fiber is the intersection of a ray with a
convex set (see again Figure \ref{fig:retract}), hence is contractible.  
Our claim then follows from the main result of \cite{Sma57}.
\end{proof}

\begin{lemma}
\label{lem:incomparables}
If $H\in\A$ and $X\in L(\A)$, and $H\not\leq X$, then the composite
$r_X\circ i_X\circ r_H$ is null-homotopic.
\end{lemma}

\begin{proof}
We will continue with the notation above.  It is enough to 
show that image of the restriction $r_X\circ i_X\mid_{P_H}\colon 
P_H\to M(\A)$ is
contractible.  Recall that $P_H=M(\A_H)\cap B_H$, and we claim that
$r_X\circ i_X(P_H)=r_X\circ i_X(B_H)$.

As in the previous lemma, it is enough to examine a fiber
$(r_X\circ i_X)^{-1}(x)$ for a point $x\in \partial P_X$.  Let $R_x$ denote
the ray through $x_X$ and $x$.  By hypothesis, $x_X\not\in H$ and 
$x\not\in H$, so $R_x\cap H$ consists of at most one point.  Note that 
$R_x\cap (B_H\setminus P_H)=R_x\cap H$.  By slightly varying the 
radius $\epsilon$ in the definition of $B_H$, we may assume
that if $\abs{R_x\cap B_H}=1$, then $R_x\cap B_H\cap H=\emptyset$. 
If $\abs{R_x\cap B_H}>1$, then $R_x\cap B_H=R_x\cap P_H$, 
which proves the claim. 

By \cite{Sma57} again, the image of $r_X\circ i_X$
is homotopy equivalent to the contractible set $B_H$.
\end{proof}

We now consider the restriction of the Hopf fibration to 
the complement of $\A_X$.  Let $j_X\colon \C^*\to M(\A_X)$ be
the inclusion of the fiber containing the basepoint $x_V$.  By definition,
for $X\leq Y$, we have $j_X=i_{X,Y}\circ j_Y$.  
Recalling the definition $\gamma_0^X=(j_X)_\sharp(1)$ 
for each $X\in L(\A)$, we see that for any $X\leq Y$, 
\begin{equation}
\label{eq:projections}
\gamma_0^X=(i_{X,Y})_\sharp(\gamma_0^Y).
\end{equation}
In particular, $\gamma_0^X=(i_X)_\sharp(\gamma_0)$, for all $X$, and 
by \eqref{eq:gammax}, we have 
\begin{equation}
\label{eq:gamx}
\gamma_X=(r_X\circ i_X)_\sharp(\gamma_0).
\end{equation}
Likewise, in homology, we have $a_X=(r_X\circ i_X)_*([\gamma_0])=
(r_X\circ i_X\circ j)_*(1)$.

For each $X\in L(\A)$, let us define an endomorphism of 
$H_1(M(\A),\Z)$, 
\begin{equation}
\label{eq:px}
p_X=(r_X\circ i_X)_*. 
\end{equation} 
By Lemma~\ref{lem:retract}, we have that 
$p_X^2=p_X$, for each $X$.

\begin{proposition}
\label{prop:projectors}
For all $H\in \A$ and $X\in L(\A)$, we have
\[
p_X(a_H)=\begin{cases}
a_H & \text{if $H\leq X$;}\\
0 & \text{otherwise.}
\end{cases}
\]
\end{proposition}

\begin{proof}
If $H\leq X$, we write $a_H=p_H(j_*(1))$ and calculate:
\begin{eqnarray*}
p_X(a_H)&=&(r_X\circ i_X\circ r_H \circ i_H\circ j)_*(1)\\
&=&(r_{X,0}\circ r_{H,X} \circ i_H\circ j)_*(1)
\quad\text{by \eqref{eq:defrxy};}\\
&=& (r_H\circ i_H\circ j)_*(1) \quad\text{by Lemma~\ref{lem:rfunctorial};}\\
&=& a_H.
\end{eqnarray*}
If $H\not\leq X$, by Lemma~\ref{lem:incomparables}, we have
$(r_X\circ i_X\circ r_H)_*=0$, and the conclusion follows.
\end{proof}

The goal of all this was to conclude:

\begin{corollary}
\label{cor:axah}
For each flat $X\in L(\A)$, the following relation holds
\begin{equation}
\label{eq:ex}
a_X=\sum_{H\in\A\colon H\leq X} a_H.
\end{equation}
\end{corollary}

\begin{proof}
Using Proposition~\ref{prop:projectors} when $X=H$ is a hyperplane, 
we see that $p_H(a_K)=\delta_{H,K} a_K$ for $K\in\A$, so
$\sum_{H\in\A}p_H$ acts as the identity map on $H_1(M(\A),\Z)$.  Thus 
\begin{eqnarray*}
a_X &=& p_X([\gamma_0]) \\
&=& p_X\circ\sum_{H\in\A}p_H([\gamma_0])\\
&=& p_X\Big(\sum_{H\in\A} a_H\Big)\\
&=& \sum_{H\leq X}a_H,
\end{eqnarray*}
using Proposition~\ref{prop:projectors} again.
\end{proof}

\subsection{Wonderful compactifications}
\label{subsec:dcp}
In their seminal paper \cite{dCP95}, De Concini and Procesi constructed  
natural compactifications of complements of hyperplane arrangements, 
as follows.   

As before, let $\A$ be an arrangement in $\C^n$, with complement $M(\A)$, 
and let $\bar{\A}$ be the corresponding arrangement in $\PP^{n-1}$, 
with complement $U(\A)$.  A {\em building set}\/ for the arrangement is a subset 
\begin{equation}
\label{eq:building set}
\G\subseteq L_{\geq1}(\A)
\end{equation}
which indexes a choice of subspaces that are iteratively blown up in 
their construction: see \cite{De14} for details, definitions and references.
The wonderful compactification, $Y_\G$,
is a smooth, complete projective variety for which 
$D:=Y_\G\setminus U(\A)$ is a normal crossings divisor.  
The components of $D$ are indexed by the building set $\G$: write
\begin{equation}
\label{eq:div}
D=\bigcup_{x\in \G}D_x.
\end{equation}

We recall that, for a subset $S\subseteq\G$, the intersection
$D_S:=\bigcap_{x\in S}D_x$ is nonempty if and only if $S$ is a {\em nested set}.
If $\G=L_{\geq1}(\A)$, we note that nested sets are simply chains in 
$L(\A)\setminus\set{\hat{0},\hat{1}}$.  We also recall that 
nested sets form a simplicial complex on the set $\G$ and the
maximal element in $L(\A)$ is a cone vertex.  Let
$\N(\G)$ denote the link of the cone point, a simplicial complex of
dimension $n-2$ with vertex set $\G\setminus \set{\hat{1}}$.  

\begin{theorem}
\label{thm:local central}
Let $S\in N(\G)$ be a nested set.  Then the group 
$\bar{C}_S:=\langle \bar\gamma_{X}\colon X\in S\rangle$
is a free abelian subgroup of $\bar{G}=\pi_1(U(\A))$, 
of rank $k:=\abs{S}$.
\end{theorem}

\begin{proof}
Recall from \S\ref{subsec:h1 gens} that $\bar{G}$ is the 
quotient of  $G=\pi_1(M(\A))$ by the central subgroup 
$\Z=\langle \gamma_0\rangle$.   Thus, 
it is enough to prove that the group 
$C_S:=\langle \gamma_{X}\colon X\in S\cup\set{0}\rangle$ 
is a free abelian subgroup of $G$, of a rank $k+1$.
First, we check that if $X,Y\in S$, then $\gamma_X$ and $\gamma_Y$
commute.  If $X\leq Y$, this follows from the discussion above,
since $\gamma_Y$ is the central element in $G_Y$.  If $X$ and
$Y$ are incomparable, by the nested set property, then there is a 
linear diffeomorphism
\begin{equation}
\label{eq:maxmay}
M(\A_X)\times M(\A_Y)\cong M(\A_{X\vee Y}).
\end{equation}
On the level of fundamental groups, $\gamma_X$ and $\gamma_Y$ may
be identified with central elements in $G_X$ and $G_Y$, respectively,
so they commute as well.

To check that $C_S$ is free abelian, we note that 
the image of $\gamma_{X}$ in $H_1(M(\A),\Z)=G^{\ab}$ is $a_{X}$, 
from \eqref{eq:ex}.  The elements 
\begin{equation}
\label{eq:axs}
\set{a_X\colon X\in S\cup\set{0}}\subseteq H_1(M(\A),\Z)
\end{equation}
are linearly independent, by \cite[Prop.~5.2]{FY04}.  
Hence, the restriction of the 
abelianization map gives an epimorphism $C_{S}\to\Z^{k+1}$.
Since $C_S$ is abelian and has only $k+1$ generators, this 
map must be an isomorphism.
\end{proof}

Let $U_S$ be the intersection of an $\epsilon$-ball around $D_S$ in 
the wonderful compactification $Y_{\G}$ with the hyperplane complement 
$U(\A)$.  

\begin{corollary}
\label{cor:uinclusion}
The inclusion $i_S\colon U_S \inj U(\A)$
induces an injective map of fundamental groups.
\end{corollary}

\begin{proof}
Note that $U_S\simeq(\C^*)^k$, a product of $\C^*$-orbits passing 
through $x_V$, indexed by the elements $X\in S$.  Using formula 
\eqref{eq:gamx}, we see that $\pi_1(U_S,\bar{x}_V)\cong \Z^k$ and 
the homomorphism $(i_S)_{\sharp}\colon \pi_1(U_S,\bar{x}_V)\to \pi_1(U(\A),\bar{x}_V)$ 
has image the free abelian group  $\bar{C}_S\cong \Z^k$.  
The conclusion follows from Theorem \ref{thm:local central}.
\end{proof}

\subsection{Spectral sequences from the standard embeddings}
\label{subsec:ss hyp}

For a flat $X\in L(\A)$, following the standard matroid-theoretic
terminology, we say that $X$ is {\em connected}\/ if $\beta(\A_X)\neq0$, 
where  $\beta(\A):=\chi(U(\A))$. 
This is equivalent to the condition that the closed subarrangement
$\A_X$ is {\em irreducible} (see \cite{Cr67}).  In the context of
cohomological vanishing results, some authors have also referred 
to such flats as ``dense edges.''  We note, in particular, that
hyperplanes are connected flats.  

The construction from \cite{DJLO11} is obtained as follows.  If
$\A$ is a (central) arrangement in $\C^n$, choose a hyperplane $H_0\in\A$.
Let $Y=\C^{n-1}$, regarded as the affine chart of $\PP^{n-1}$ with 
$H_0$ at infinity, and let $W=U(\A)$.  The indexing poset 
$L(W,Y)$ can be identified with 
\begin{equation}
\label{eq:ph0}
P_{H_0}:=\set{X\in L^{\opp}(\A)\colon H_0\not\geq X}.
\end{equation}
The subsets $X$ are affine spaces, hence contractible.  Applying 
Corollary \ref{cor:ss mdf arr}, we obtain the following result.

\begin{corollary}
\label{cor:ss hyp arr}
Let $\A$ be a central hyperplane arrangement in $\C^n$.  For any locally 
constant sheaf $\F$ on the  projectivized complement $U=U(\A)$, 
we have a spectral sequence 
\begin{equation*}
\label{eq:ss ph0}
E^{pq}_2=\prod_{X\in P_{H_0}}
\widetilde{H}^{p-\rho(X)-1}(\abs{(P_{H_0})_{<X}};H^{q+\rho(X)}(U,\F_{U_X})) 
\Rightarrow H^{p+q}(U,\F).
\end{equation*}
\end{corollary}

If $\A$ is a (central) arrangement and instead we choose $Y=\PP^{n-1}$ and
let $U=U(\A)$ be the complement of $\bar\A$, then the indexing poset of 
$D=Y\setminus U$ is the
truncated intersection poset, $P(\A):= L^{\opp}(\A)\setminus\set{0}$.
The subsets $\bar{X}$ in the decomposition of $Y$ are 
projective linear spaces:  here, $\rho(X)$ equals the complex 
dimension of $\bar{X}$, and $\phi^{-1}(P(\A)_{\leq X})$ is
homotopic to the projective space $\bar{X}$.  We have
a combinatorial cover, and 
$D_X\simeq \phi^{-1}(P(\A)_{<X})$ is the union of hyperplanes in the
(projective) arrangement $\overline{\A^X}$, where recall 
$\A^{X} = \set{H\cap X \mid H\in \A\setminus \A_X}$. 
For any locally constant
sheaf $\F$, our spectral sequence has the form
\begin{equation}
\label{eq:projective_arr}
E^{pq}_2=\prod_{X\in P(\A)}
H^{p-\rho(X)}_c \big(U(\A^X); 
H^{q+\rho(X)}(U,\F_{U_X})\big)\Rightarrow
H^{p+q}(U,\F).
\end{equation}

We remark that the cohomology local system in \eqref{eq:projective_arr} need
not be trivial, in contrast with the spectral sequence in 
\cite{DJLO11,DO12,DS13}, where the local system is necessarily trivial.

\begin{example}
\label{ex:three lines}
If $\A$ consists of $3$ distinct lines in $\C^2$, then 
$U=\PP^1\setminus \set{\text{three points}}$.  Then
\[
P=\begin{tikzpicture}[scale=0.5,baseline=(current bounding box.center)]
\node[circle,draw,inner sep=1pt] (0) at (0,0) 
             [label=above:$\PP^1$] {};
\node[circle,draw,inner sep=1pt] (1) at (-1,-1)
       [label=below:$1$] {};
\node[circle,draw,inner sep=1pt] (2) at (0,-1)
       [label=below:$2$] {};
\node[circle,draw,inner sep=1pt] (3) at (1,-1)
       [label=below:$3$] {};
\draw (0) -- (1);
\draw (0) -- (2);
\draw (0) -- (3);
\end{tikzpicture}.
\]
Noting that $F_2=\pi_1(U,*)$, let $\F$ be the regular representation $\k[F_2]$. 
The factors of \eqref{eq:projective_arr} consist of $H^{p-1}_c(U,\k[F_2])$,  
indexed by $\PP^1$, for $q=0$, and three factors $H^q(S^1,\k[F_2])$
indexed by points, for $p=0$.  Then $E^{pq}_2=E^{pq}_3$ have two nonzero
entries, joined by the differential $d_3^{01}$, and the unique nonzero 
cohomology group $H^1(U,\k[F_2])=H^1(F_2,\k[F_2])$ is its kernel.
\end{example}

\subsection{A spectral sequence from the wonderful compactification}
\label{subsec:wonderful}
Another combinatorial cover of an arrangement complement 
can be constructed using the De Concini--Procesi wonderful 
compactification \cite{dCP95}. 

Accordingly, Theorem~\ref{thm:cc} provides a combinatorial cover 
of $U(\A)$.  This time, the indexing poset is the opposite of the
face lattice of $\N(\G)$, which we denote $P_\G$.  For a nested 
set $S\in \N(\G)$, the (complex) dimension of the
intersection $D_S:=\bigcap_{x\in S}D_x$ is $n-1-\abs{S}$: for ease of
notation, we shift the indexing by $n-1$ and let $\rho(S)=-\abs{S}$.

The key advantage of this choice of combinatorial cover is that
the open sets $U_S$ are complements of $\abs{S}$ smooth components 
that intersect locally transversally.  That is, $U_S\simeq (S^1)^{\abs{S}}$, 
a torus.  

The pairs $(\phi^{-1}((P_{\G})_{\leq S}),\phi^{-1}((P_{\G})_{< S}))$
are pairwise homotopy equivalent to $(D_S,\bigcup{D_{S'}})$ for faces
$S'\supsetneq S$ in $\N(\G)$.  For a nested set $S$ and an element $X\in S$,
let us abbreviate
\begin{equation}
\label{eq:sx}
\bigvee S_{<X}:=\bigvee_{Z\in S\colon Z<X}Z.
\end{equation}
Using \cite[\S 4.3]{dCP95}, one can deduce that 
\begin{equation}
\label{eq:ws}
U^S := D_S\setminus \bigcup_{S'\supsetneq S} D_{S'}=\prod_{X\in S}
U\Big(\A_X^{\bigvee S_{<X}}\Big).
\end{equation}
The ranks of the arrangements $\A_X^{\bigvee S_{<X}}$ add up to $n-1$, 
again by \cite[\S 4.3]{dCP95}, so we see that $U^S$ has the 
homotopy type of an $(n-1-\abs{S})$-dimensional cell complex.  

Applying again Corollary \ref{cor:ss mdf arr}, we obtain yet another spectral 
sequence converging to the cohomology of the projectivized complement 
$U=U(\A)$, with coefficients in a locally constant sheaf $\F$.
Since each intersection $D_S$ is a compact, we will express 
the $E_2$-page in terms of compactly supported cohomology 
(cf.~Remark~\ref{rem:xdx}).

\begin{corollary}
\label{cor:ss won}
With notation as above, we have a spectral sequence
\begin{equation*}
E_2^{pq}=\prod_{S\in \N(\G)} H_c^{p+\abs{S}}(U^S;
H^{q-\abs{S}}((S^1)^{\abs{S}},\F_{U_S}))\Rightarrow H^{p+q}(U,\F).
\end{equation*}
\end{corollary}

\section{Vanishing of cohomology}
\label{sect:hyparr-vanish}

The main objective of this section is to give conditions for the cohomology
of a local system on an arrangement complement to vanish, in terms of 
certain free abelian subgroups of the fundamental group.  These subgroups
are fundamental groups of tori embedded in the arrangement complement,
and they arise as neighbourhoods around intersections of boundary 
components in the wonderful compactification of the previous section.

\subsection{Local systems on tori}
\label{subsec:tori}
First, then, we recall the commutative algebra which describes the cohomology 
of local systems on a torus.  Let $\k=\Z$ or a field. If $I$ is an ideal in a 
commutative, Noetherian ring $R$ and $A$ is an $R$-module, we say 
a sequence of elements $g_1,\ldots,g_k\in $ forms a {\em regular sequence}\/ 
for $A$ if $g_i$ is a non-zero divisor on 
$A/(g_1,\ldots,g_{i-1})A$, for $1\leq i\leq k$.  

We say that $\depth(I,A)=\infty$ if there exists a regular sequence 
$g_1,\ldots,g_k$ for which 
$A/(g_1,\ldots,g_k)A=0$; otherwise, the depth is the length of the longest 
regular sequence for which $A/(g_1,\ldots,g_k)A\neq 0$.  If $(R,\m)$ is a 
local ring, $A$ is a {\em maximal Cohen--Macaulay module} when 
$\depth_R(\m,A)=\dim R$.  Depth has a well-known homological interpretation:

\begin{proposition}
\label{prop:depth,Ext}
For any $R$-module $A$ which is not necessarily finitely-generated,
$\Ext^p_R(R/I,A)=0$ for all $p<\depth_R(I,A)$.
\end{proposition}

This translates directly to a cohomological vanishing condition for local
systems on tori.  

\begin{lemma}
\label{lem:torus}
Suppose $A$ is a local system on $(S^1)^n$, for some $n\geq1$: 
that is, a $R:=\k[\Z^n]$-module.
Let $\one$ denote the augmentation ideal in $R$.
Then $H^n((S^1)^n,A)\cong A/\one A$, and
$H^p((S^1)^n,A)=0$ for all $p\leq \depth_R(\one,A)$.
\end{lemma}

\begin{proof}
We note that $H^p((S^1)^n,A)=H^p(\Z^n,A)=\Ext_R^p(\k,A)$, and we 
apply Proposition~\ref{prop:depth,Ext}.
\end{proof}

\begin{remark}
\label{rem:cm module}
Since the $R$-module $\Ext^p_R(\k,A)$ is supported at most at the
augmentation ideal $\one$, we may localize to note that
$H^p((S^1)^n,A)=0$ if and only if $\Ext^p_{R_{\one}}(\k,A_{\one})=0$.
above depends only on the localization $A_\one$.  That is, 
\begin{equation}
\label{eg:htorus}
H^p((S^1)^n,A)=0\:\:\Leftrightarrow\:\: p\neq n, 
\end{equation}
if and only if $A_\one$ is a maximal Cohen--Macaulay 
module over $R_\one$.
\end{remark}

\begin{definition}
\label{def:MCM}
With this in mind, we will say that a module $A$ over $R:=\k[\Z^n]$
is a {\em $\MCM$-module}\/ if $\depth_R(\one,A)\geq n$.  Equivalently,
the localization $A_{\one}$ is either zero or a maximal Cohen--Macaulay
module over $R_{\one}$.
\end{definition}

We will frequently consider $\MCM$-modules which arise as follows.
\begin{proposition}\label{prop:MCMcriteria}
Suppose that $H$ is a free abelian subgroup of a discrete group $G$.
Then $\k[G]$ and $\k[G^{\ab}]$ are both $\MCM$-modules over $\k[H]$.
\end{proposition}
\begin{proof}
The group algebra $\k[G]$ is a free module over $\k[H]$, and a free
module over a Cohen--Macaulay ring is maximal Cohen--Macaulay.  Since
the natural map $H\to G^{\ab}$ is injective, the same argument
applies to $\k[G^{\ab}]$.
\end{proof}

\subsection{Cohomology with twisted coefficients}
\label{subsec:coho twisted} 

We are now ready to state and prove the main result 
of this section.  

In \S\ref{subsec:wonderful}, we constructed a combinatorial cover
of $U(\A)$ using its embedding in a De~Concini--Procesi 
compactification.  This leads to the following vanishing result.

For an arrangement $\A$, fix a building set $\G$.  For each nested
set $S\in \N(\G)$, let $R_S=\k[C_S]$.
This is a ring of Laurent polynomials of dimension $\abs{S}$, 
by Theorem~\ref{thm:local central}. 
We continue to let
$\one=\one_S$ denote the augmentation ideal in $R_S$.

\begin{theorem}
\label{thm:arrvanish}
Let $\A$ be a central, essential arrangement of hyperplanes in $\C^n$,
and $U=U(\A)$ its projective complement.  
Let $A$ be a local system on $U$ over a principal ideal domain
$\k$.  Suppose that $A$ is a $\MCM$-module over $R_S$ for each $S\in \N(\G)$.  
Then $H^p(U,A)=0$ for all $p\neq n-1$.  If, moreover, $A/\one_S A$ is
free over $\k$ for each nested set $S$, then so is $H^{n-1}(U,A)$.
\end{theorem}

\begin{proof}
Let $\F$ denote the locally constant sheaf corresponding to $A$.
By Corollary~\ref{cor:ss won} together with Lemma~\ref{lem:torus},
we have a spectral sequence
\begin{equation}\label{eq:wonderfulE2}
E_2^{pq}=\prod_{S\in \N(\G)} H_c^{p+\abs{S}}(U^S;
\Ext^{q-\abs{S}}_{R_S}(\k,A))\Rightarrow H^{p+q}(U,\F).
\end{equation}
Since $U^S$ is a product of arrangement complements, it is a Stein manifold, and
Corollary~\ref{cor:cpct vanishing trick} shows compactly supported cohomology of
$U^S$ in a locally constant sheaf is zero unless $p+\abs{S}\geq
n-1-\abs{S}$: that is, $p\geq n-1-2\abs{S}$.

By Definition~\ref{def:MCM}, we have that 
$\depth_{R_S}(\one,A)\geq \dim R_S=\abs{S}$ for each ring $R_S$: that is, 
$\Ext^{q-\abs{S}}_{R_S}(\k,A)=0$ for $q-\abs{S}\neq\abs{S}$:
i.e., $q=2\abs{S}$.  By Lemma~\ref{lem:torus} again, 
$\Ext^{\abs{S}}_{R_S}(\k,A)=A/\one_S A$.

We see that $E_2^{pq}=0$ unless $p+q\geq n-1$, so $H^p(U,\F)=0$ 
except possibly for $p=n-1$.  If, moreover, each entry $E_2^{p,n-1-p}$
is free, then so are the submodules $E_\infty^{p,n-1-p}$, and 
$H^{n-1}(U,\F)$ as well.
\end{proof}

In Proposition~\ref{prop:MCMcriteria}, we saw that $A=\k[G]$ and
$A=\k[G^{\ab}]$ satisfy the hypotheses of the theorem.  Other examples
can be given in terms of conditions on the group action.

\begin{lemma}
\label{lem:fp free}
Suppose $A$ is a $\k[G]$-module which either either semisimple or 
finite-dimensional over a field $\k$.  If $A^{\gamma_X}=0$ for some 
$X\in \G$, then $A$ is a $\MCM$-module over $R_S$ for each 
nested set $S\in\N(\G)$ containing $X$.
\end{lemma}

\begin{proof}
Consider the exact sequence
\begin{equation}\label{eq:invariants}
\xymatrix{
0\ar[r] & A^g\ar[r] & A\ar[r]^{1-g}\ar[r] & A\ar[r] & A_g\ar[r] & 0.
}
\end{equation}

With either hypothesis, we see $A^g\cong 0$ if and only $A_g\cong 0$, 
for any $g\in G$.  That is, if $A^g=0$, then $1-g$ acts by
an automorphism on $A$.  By functoriality, $1-g$ induces an automorphism on
$\Ext^\cdot_{R_S}(\k, A)$ as well.  On the other hand, $1-g$ acts as zero
on $\k$, hence on $\Ext^\cdot_{R_S}(\k, A)$.  Combining the two, we see
$\Ext^p_{R_S}(\k, A)=0$ for all $p$.
\end{proof}

This special case gives a slight generalization of
 Kohno's classical vanishing result from \cite{Ko86}:

\begin{corollary}
\label{cor:fd vanishing}
Let $\A$ be a central, essential complex arrangement, and let $U=U(\A)$.  
Suppose $A$ is a $\k[\pi_1(U)]$-module which is either semisimple or 
finite-di\-mensional over a field.  If $A^{\gamma_X}=0$ for all connected 
flats $X\in L_{\geq 1}(\A)$, then $H^p(U,A)=0$ for all $p\neq n-1$.
\end{corollary}

\begin{proof}
We use Theorem~\ref{thm:arrvanish}.  Let $\G$ be the set
of all connected, nonzero flats (the minimal building
set, in the language of \cite{dCP95}).  Suppose $S\subseteq \G$ is a nonempty
nested set and $X\in S$.  Using the hypothesis together with 
Lemma~\ref{lem:fp free}, we see the only nonzero factor of the
spectral sequence
\eqref{eq:wonderfulE2} is indexed by $S=\emptyset$.

In that case, $W_\emptyset=U$, $R_\emptyset=\k$, and $E^{pq}_2=0$ 
unless $q=0$ and $p=n-1$, which implies $H^p(U,A)=0$, unless $p=n-1$.
\end{proof}

\subsection{Weaker vanishing results}
\label{subsec:variation} 
Essentially the same arguments gives more information even when the
conditions on the $\k[\pi_1(M(\A))]$-module $A$ are relaxed.  Here,
we briefly consider weaker conditions on $A$.

Let $X$ be a connected space that has the homotopy type of a 
finite-dimensional CW-complex, and let $n=\gd(X)$ be the minimal 
dimension of such a CW-complex.  If $A$ is a $\pi_1(X)$-module, 
we define 
\begin{equation}
\label{eq:ddef}
\dep(X;A)=\max \{ i : H^{n-i}(X,A)\neq 0 \}, 
\end{equation}
where by convention $\dep(X;A)=-1$ if $H^\hdot(X,A)=0$.
Then Theorem~\ref{thm:arrvanish} admits the following generalization.

\begin{proposition}
\label{prop:depth}
Let $A$ be a local system on a projective arrangement complement 
$U(\A)$ of dimension $n-1$.  Let $\G$ be a building set for $\A$.  Then
\begin{equation}\label{eq:lower bound}
\dep(U(\A);A)\leq \max_{S\in N(\G)}\abs{S}-\depth_{R_S}(\one_S,A).
\end{equation}
\end{proposition}

\begin{proof}
Consider the spectral sequence \eqref{eq:wonderfulE2} of 
Theorem~\ref{thm:arrvanish}, and a nested set $S$.  
Once again, the compactly supported cohomology of $U^S$ 
is zero unless $p\geq n-1-2\abs{S}$.  By Proposition~\ref{prop:depth,Ext},
the coefficient module is zero for $q-\abs{S}\geq \depth(\one_S,A)$,
so $E^{pq}_2=0$ for
\[
p+q< n-1-\abs{S}+\depth(\one_S,A).
\]

Taking the minimum of the right-hand side over $S\in N(\G)$ gives a 
lower bound on the degree of the first nonzero cohomology group, which
translates to the inequality \eqref{eq:lower bound}.
\end{proof}

The following special case includes rank-$1$ local systems over a field,
and uses the same argument as Corollary~\ref{cor:fd vanishing}.

\begin{corollary}
\label{cor:fd nonvanishing}
Suppose a $\k[\pi_1(U)]$-module $A$ is semisimple or finite-dimensional 
over a field.  
Then $\dep(U(\A);A)$ is at most the largest cardinality of a nested set
$S$ having the property that $A^{\gamma_X}\neq0$, for all $X\in S$.
\end{corollary}

\section{Elliptic arrangements}
\label{sect:tori}

\subsection{Abelian arrangements}
\label{subsec:abel arrs}

An abelian arrangement is a finite set of codimension one abelian 
subvarieties (possibly translated) in a complex abelian variety, 
see \cite{LV12, Bi13}. 
Our construction gives a combinatorial cover of the complement 
of such an arrangement, and thus, a spectral sequence converging 
to the cohomology of the complement.    
A special case of the construction is the configuration 
space of $n$ points on an
elliptic curve $E$, which we consider in \S\ref{subsec:configs torus}.  

The examples we consider are arrangements in a product
of complex elliptic curves, $E^{\times n}$.  For short, we will refer to
these as {\em elliptic arrangements}. Let $\A=\{H_1,\dots , H_m\}$ be 
such an arrangement. 
The subvarieties $H_i\subseteq E^{\times n}$ are fibers of group 
homomorphisms.  Such homomorphisms are parameterized by 
integer vectors: writing $E$ as an additive group, 
we may write $H_i=f_i^{-1}(\zeta_i)$ for some point 
$\zeta_i\in E$, where
\begin{equation}
\label{eq:ellpoly}
f_i(x_1,\ldots,x_n)=\sum_{j=1}^n a_{ij} x_j,
\end{equation}
and $A=(a_{ij})$ is a $m\times n$ integer matrix. We let $\corank(\A):=n-\rank(A)$.

The connected components of 
elements of the intersection lattice of $\A$ are cosets of products 
of elliptic curves.  In particular, the intersection $\Sigma=\bigcap_{H\in\A}H$ 
is either empty or homeomorphic to a (finite) disjoint union of copies 
of $E^{\times r}$, where $r$ is the corank of $\A$.
We say $\A$ is {\em essential} if its corank is zero: 
clearly, in this case, $m\geq n$.  

\begin{proposition}
\label{prop:elliptic Stein}
Let $\A$ be an essential
elliptic arrangement.  Then the complement $U(\A)$ is a Stein manifold.
\end{proposition}

\begin{proof}
Since $\A$ is essential, there exists an $n\times n$ minor of $A$
which is invertible over $\Q$.  Without loss, we may assume it is the
first $n$ rows, and consider the corresponding subarrangement.
Since the complement of an analytic set in a Stein manifold is again
Stein, it is sufficient to prove that the complement of the subarrangement
is Stein: that is, assume $A$ is a square matrix.

We proceed by induction on the integer value $\abs{\det{A}}$.  If it is 
one, then $A$ is invertible over $\Z$, in which case let 
$g\colon E^{\times n}\to
E^{\times n}$ be the automorphism given by the inverse matrix.  
Then $(f_i\circ g)(x_1,\ldots,x_n)=x_i$ for $1\leq i\leq n$, and so 
\begin{align}
U(\A) &= E^{\times n}\setminus \bigcup_{1=1}^n H_i\\
\notag
& \cong \prod_{i=1}^n (E\setminus \set{\zeta_i}).
\end{align}

In this case, the claim follows by some classical facts about 
Stein manifolds (see for instance \cite{FG02}): 
an open Riemann surface is Stein, and 
the Stein property is closed under products.

If $\abs{\det{A}}>1$, using the Hermite Normal Form via a 
change of variables as above, we may assume $a_{ij}=0$ 
for $i>j$ and $a_{ij}<a_{ii}$ for $j<i$.  By hypothesis, we 
have $a_{ii}>1$ for some diagonal entry $1\leq i\leq n$.
Reordering rows and columns, we may take $i=n$.

Consider the arrangement $\A'=(\A\setminus \set{H_n}) \cup\set{H'_n}$, 
where $H'_n$ is given by the equation 
$x_n+\sum_{j=1}^{n-1}a_{nj}x_j=\zeta_n$.  

Now let $g\colon E^{\times n}
\to E^{\times n}$ be given by letting $g_i=\id$ for $1\leq i<n$ and 
$g_n(x)=a_{nn}x$.  The restriction of $g$ to $U(\A)$ is a
finite, holomorphic map whose image is $U(\A')$.  On the other hand,
the elliptic arrangement $\A'$ is essential, but the determinant of its
coefficient matrix is strictly smaller, so $U(\A')$ is Stein by induction,
and $U(\A)$ is Stein by \cite[Prop.~V.1.1]{FG02}.
\end{proof}

\begin{corollary}
\label{cor:elliptic dim}
The complement $U(\A)$ of an elliptic arrangement $\A$ in $E^{\times n}$
has the homotopy type of a CW-complex of dimension $n+r$, where
$r=\corank(\A)$.
\end{corollary}

\begin{proof}
If $\A$ is essential, then $U(\A)$ is Stein, and thus has the homotopy type of a
$n$-dimensional CW-complex.  Otherwise, $U(\A)\cong U(\A')\times E^{\times r}$,
where $\A'$ is an essential arrangement in $E^{\times(n-r)}$.
\end{proof}

The previous section gave a cohomological vanishing result for hyperplane
arrangements.  The corresponding result for elliptic arrangements can 
be obtained analogously.  Let $\A$ be an elliptic arrangement, with 
intersection poset $P(\A)$.  As noted previously, the connected components 
of intersections of elliptic hyperplanes are cosets of products of elliptic curves.  
For any $X\in P(\A)$, let $T\A_X=\{TH\}_{H\in \A\colon X\subseteq H}$ 
be the arrangement of tangent hyperplanes to the arrangement at 
the coset $X\subseteq E^{\times n}$.  
Now the open set $U_X$ for which $m(U_X)=X$ is a strong deformation 
retract of the (linear) hyperplane arrangement complement $U(T\A_X)$, 
and the restriction $\A^X$ is again an elliptic arrangement.

\begin{theorem}
\label{thm:elliptic vanishing}
Let $\A$ be an essential elliptic arrangement of dimension $n$.
Suppose that $A$ is a $\k[\pi_1(U(\A))]$-module such that, for 
each $X\in P(\A)$, the module $A$ satisfies the hypotheses of 
Theorem~\ref{thm:arrvanish} for the space $U(T\A_X)$.
Then $H^i(U(\A),A)=0$ unless $i=n$.  If, moreover, 
$H^\hdot(U(T\A_X),A)$ is a free $\k$-module for each $X$, then
so is $H^n(U(\A),A)$.
\end{theorem}

\begin{proof}
We use Corollary~\ref{cor:ss mdf arr}.  
Let $\F$ denote the locally constant sheaf corresponding to $A$. 
Since the ambient space is compact, we have a spectral sequence with
\begin{equation}
\label{eq:ssfbar}
E^{pq}_2=\prod_{X\in P(\A)}
H_c^{p-\rho(X)}(U(\A^X),\,
H^{q+\rho(X)}(U(T\A_X),\F_{U_X}))
\end{equation}
which converges to $H^{p+q}(U(\A),\F)$.
By hypothesis, the coefficient local system is zero unless 
$q+\rho(X)=n-\rho(X)$, since $U(T\A_X)$ is a central arrangement of rank
$n-\rho(X)$.

On the other hand, the complement is a Stein manifold 
(Proposition~\ref{prop:elliptic Stein}), so by
Corollary~\ref{cor:cpct vanishing trick}, the term indexed by $X$
also vanishes unless $p-\rho(X)=\rho(X)$.  Combining,
we see that $E_2^{pq}=0$ unless $p+q=n$.  The last assertion 
follows as in Theorem~\ref{thm:arrvanish}.
\end{proof}

\begin{corollary}
\label{cor:ellipticarr duality}
Let $\A$ be an elliptic arrangement in $E^{\times n}$.  
Then its complement, $U(\A)$, is both a duality space and an abelian 
duality space of dimension $n+r$, where $r$ is the corank of $\A$. 
\end{corollary}

\begin{proof}
If $\A$ is essential, then the assertions follow from the previous theorem:
if $G=\pi_1(U(\A))$ then $\k[G]$ and $\k[G^{\ab}]$ both satisfy
the hypotheses of Theorem~\ref{thm:arrvanish}
for each tangent hyperplane arrangement $T\A_X$, by
Proposition~\ref{prop:MCMcriteria}.

Otherwise, $U(\A)\cong U(\A')\times E^{\times r}$, where $\A'$ is an 
essential arrangement in $E^{\times (n-r)}$.  Since $E^{\times r}$ 
is both a duality space and an abelian 
duality space of dimension $2r$, we are done. 
\end{proof}
At the other end of the spectrum, we also obtain a vanishing result for
rank-$1$ local systems, comparable to Corollary~\ref{cor:fd vanishing}:
\begin{corollary}
Let $\A$ be an essential elliptic arrangement, and let $U=U(\A)$.  
Suppose $A$ is a $\k[\pi_1(U)]$-module which is either semisimple or 
finite-dimensional over a field.  Suppose further that
$A^{\gamma_X}=0$ for all $X\in P(\A)$ for which the hyperplane
arrangement $T\A_X$ is irreducible. Then $H^p(U,A)=0$ for all $p\neq n-1$.
\end{corollary}

\subsection{Convenient local systems}
\label{ss:LV12}
Our approach also generalizes a vanishing result of Levin and 
Varchenko, \cite[Thm.~5.2]{LV12}.  In the language of
forms, these authors consider certain complex, rank-$1$ local systems
on the complement of an elliptic arrangement, which we outline briefly
here.  

Levin and Varchenko consider local systems on $U=U(\A)$ obtained 
by restricting a rank-$1$ local system $\C_\rho$ on the ambient space, 
$E^{\times n}$.  They impose then a condition on $\C_\rho$, which they call 
{\em convenient}, under which it is the case that $H^p(X,\C_\rho|_X)=0$ 
for all $p$ and all positive-dimensional $X\in P(\A)$.  They conclude that 
$H^p(U,\C_\rho|_U)=0$ for all $p<n$, and give a concrete 
description of $H^n(U,\C_\rho|_U)$.  We give the following generalization
of their vanishing result from \cite{LV12}.  

We need a technical result about intervals in the intersection poset $L'(\A)$,
introduced in \S\ref{subsec:cover-arr}.  Let
\begin{equation}
\label{eq:pprime}
P'(\A)=\set{
\set{H\in\A\colon H\supseteq X}\colon \text{for some $X\in L'(\A)$}},
\end{equation}
a poset ordered by reverse inclusion.  Clearly $P'(\A)$ is isomorphic to
$L'(\A)$, by relabelling intersections by the elliptic hyperplanes that
contain them.  Let $p\colon P(\A)\to P'(\A)$ denote the order-preserving 
map defined by $p(X)=\set{H\in\A\colon H\supseteq X}$.  If $\A$ is unimodular,
then each intersection $X$ is connected, and $p$ is an isomorphism.  In
general, $p$ is merely surjective.

\begin{lemma}
\label{lem:connected interval}
For any $S\in P'(\A)$ and $X\in p^{-1}(S)$, the restriction
$p\colon (X,E^{\times n})\to P'(\A)_{> S}$ induces a homotopy equivalence of
order complexes, where $(X,E^{\times n})$ denotes an open interval in $P(\A)$.
\end{lemma}

\begin{proof}
If $T\subset S$, consider the restriction of the fiber,
$p^{-1}(P'(\A)_{\geq T})\cap (X,E^{\times n})$.  By construction, this set equals
\[
\set{Y\in P(\A)\colon X\subseteq Y\text{and $Y\subseteq H\Rightarrow
H\in T$}}.
\]
The connected component of $\bigcap_{H\in T}H$ containing $X$ is the
unique minimal element, so the fiber has a cone point.  The result follows
by Quillen's Lemma~\cite{Qu78}.
\end{proof}

\begin{theorem}
\label{thm:LV}
Let $\A$ be an essential elliptic arrangement in $E^{\times n}$, 
and $U=U(\A)$ the complement.  Let $A$ be a local system on $E^{\times n}$
for which $H^p(X,A|_X)=0$ for all $p$ and positive-dimensional $X\in P(\A)$.
Then $H^p(U,A|_U)=0$ for $0\leq p\leq n-1$.
\end{theorem}

\begin{proof}
Let $\F$ denote the locally constant sheaf with stalk $A$, and 
$\Sigma(\A)=E^{\times n}\setminus U(\A)$, the union of the elliptic 
hyperplanes.  
Let $i,j$ denote the inclusions of $U(\A)$ and $\Sigma(\A)$, respectively,
in $E^{\times n}$.
The hypotheses imply $H^*(E^{\times n},\F)=0$, so the long exact 
sequence from
\begin{equation}
\label{eq:ses}
\xymatrix@-0.8em{
0\ar[r]&i_!i^*\F\ar[r] & \F\ar[r] & j_*j^*\F\ar[r] & 0
}
\end{equation}
implies that 
\begin{equation}
\label{eq:les triple}
H^p(U(\A),\F_{U(\A)})\cong H^{p-1}(\Sigma(\A),\F_{\Sigma(\A)})
\end{equation}
for all $p\geq 0$.  By Proposition~\ref{prop:elliptic Stein}, it is 
enough to show that $H^p(\Sigma(\A),\F_{\Sigma(\A)})=0$ for $p<n-1$.  
 
We cover $\Sigma(\A)$ by $\epsilon$-neighborhoods of intersections of
elliptic hyperplanes, as in \S\ref{subsec:cover-arr}, noting that intersections
of cosets are cosets, so the elements of $L'(\A)$ have equidimensional
components.
We use the spectral sequence from Corollary~\ref{cor:ss cover-arr}, which has 
\begin{equation}
\label{eq:e2pq-ell}
E^{pq}_2=\prod_{X\in L'(\A)^{\opp}}
\tilde{H}^{p+\rho(X)-1}(\lk_{L'(\A)^{\opp}}(X);
H^{q-\rho(X)}(X,\F_X)),
\end{equation}
where $X$ is an intersection of elliptic hyperplanes.  
By hypothesis, $H^*(X,\F_{X})=0$ when $X$ is 
positive-dimensional.  If $X$ is zero-dimensional, $\rho(X)=-n$, so 
$E^{pq}_2=0$ unless $q+n=0$.  On the other hand, 
if $X$ is zero-dimensional, the link of $X$ is the realization of 
the open interval $L'(\A)_{>X}$, which is homotopic to $(X_0,E^{\times n})$ for
some connected component $X_0\in L(\A)$, by Lemma~\ref{lem:connected interval}.

On the other hand, $[X_0,E^{\times n}]\cong L(T\A_{X_0})$, an arrangement of
hyperplanes of rank $n$.  By Folkman's Theorem~\cite{Fo66}, then,
$(X_0,E^{\times n})$ has the homotopy type of a wedge of $(n-2)$-spheres,
so $\lk_{L'(\A)^{\opp}}(X)$ does as well.
It follows that $E^{pq}_2=0$ unless $p-n-1=n-2$;
i.e., $p+q=n-1$, so $E^{pq}_\infty=E^{pq}_2$, and the conclusion follows.
\end{proof}

\subsection{Configuration spaces}
\label{subsec:configs torus}
Let $M$ be a connected manifold of dimension $m$, and let $F(M,n)$ 
be the configuration space of $n$ distinct points in $M$, for $n\ge 1$.  
Let $M_{ij}=\set{x\in M^{\times n}\mid x_i=x_j}$.  Then
\begin{equation}
\label{eq:config}
F(M,n)=M^{\times n}\setminus \bigcup_{i<j} M_{ij}.
\end{equation}
If $M=\C$, the configuration space $F(\C,n)$ is a complex hyperplane
arrangement complement and a classifying space for the pure braid group.

Now we consider $M=E$, an elliptic curve, in which case $F(E,n)$ is an
elliptic arrangement complement.  
Since $E$ is a topological group, its diagonal action
on $E^{\times n}$ gives a principal $E$-bundle,
\begin{equation}\label{eq:bundle}
\xymatrix{
E\ar[r] & E^{\times n}\ar[r]^(.45){p} & E^{\times n}/E.
}
\end{equation}
Let $\Fbar(E,n)$ denote the image of $F(E,n)$ under the projection $p$.
Since $E$ is abelian, the bundle \eqref{eq:bundle} is trivial: a section
$s$ of $p$ is given by lifting a class $[x]\in E^{\times n}/E$ to
$x\in E^{\times n}$ with $x_n=1$.  By restriction, 
\begin{equation}
\label{eq:E decone}
F(E,n)\cong E\times \Fbar(E,n),
\end{equation}
where $\Fbar(E,n)$ is an essential elliptic arrangement.
Let $PE_n=\pi_1(F(E,n))$, the pure elliptic braid group 
(also known as the $n$-string pure braid group on the torus). 
The configuration space $F(E,n)$ is well-known to be a classifying space
for $PE_n$, although we note that, for $n\ge 3$, the space $F(E,n)$ is 
not formal, cf.~\cite{Be94, DPS}.

As a special case of Corollary~\ref{cor:ellipticarr duality}, we note:
\begin{corollary}
\label{cor:ell dual}
If $A=\Z[PE_n]$ or $A=\Z[PE_n^{\ab}]$, then
$H^p(PE_n,A)=0$ for $p\neq n+1$, and $H^{n+1}(PE_n,A)$ 
is a free abelian group.
Consequently, for each $n\ge 1$, the pure elliptic braid group
$PE_n$ is both a duality group and an abelian duality group 
of dimension $n$.
\end{corollary}

\begin{remark}
\label{rem:harer}
The fact that $PE_n$ is a duality group of dimension $n$ can also be deduced 
from \cite[Theorem 4.1]{Ha86}.  Indeed, it is shown there that the full elliptic 
braid group (also known as the mapping class group of the torus with $n$ marked 
points) is a virtual duality group of dimension $n$.  The claim then follows 
from the fact that $K(\pi_1(PE_n,1))=F(E,n)$ is finite-dimensional, and 
thus $PE_n$ is torsion-free.
\end{remark}

\section{Toric complexes and the Cohen--Macaulay property}
\label{sect:raag}

\subsection{Generalized moment-angle complexes}
\label{subsec:gmac}
In this last section, we illustrate our techniques on a class 
of spaces that arise in toric topology, as a basic example 
of polyhedral products.  For more background and references 
on the subject, we refer to \cite{BBCG10, Da12, DS07, PS09}.

Let $X$ be a connected, finite CW-complex, and 
let $Y$ a subspace containing a distinguished zero-cell, $*=e^0$.
Next, let $L$ be a simplicial complex on vertex set $\sV$. 
Associated to these data there is a subcomplex of the 
cartesian product $X^{\sV}$, known as a ``generalized 
moment-angle complex", and defined as 
\begin{equation}
\label{eq:zlx}
\ZZ_L(X,Y)=\bigcup_{\tau\in L}(X,Y)^\tau,
\end{equation}
where $(X,Y)^\tau=\prod_{i\in \sV} X_{\tau, i}$, and 
\begin{equation}
\label{eq:xtau}
 X_{\tau, i}= \begin{cases} X & \text{if $i\in\tau$;}\\[2pt]
Y & \text{if not.}\end{cases}
\end{equation}

Suppose now that $Y$ is open in $X$, and let $\max(L)$ denote the
set of maximal simplices in $L$.  By \eqref{eq:zlx}, the open 
sets $\Cover_L:=\set{(X,Y)^\tau\colon \tau\in \max(L)}$ cover $\ZZ_L(X,Y)$.  
If $Y\subsetneq X$ are connected manifolds, this cover is a special
case of the one constructed in \S\ref{subsec:cover-arr}.
\begin{lemma}
\label{lem:gmac cover}
Regard $\max(L)$ as a (closed) cover of $\abs{L}$. 
If $Y\subsetneq X$, then $N(\Cover_L)=N(\max(L))$.
\end{lemma}

\begin{proof}
For any $S\subseteq\max(L)$, it is easy to see that 
$\bigcap_{\sigma\in S}(X,Y)^\sigma=(X,Y)^\tau$, where 
$\tau=\bigcap_{\sigma\in S}\sigma$.  With our assumption that $X\neq Y$,
we also have $(X,Y)^\sigma=(X,Y)^\tau$ if and only if $\sigma=\tau$.  
The claim follows by combining the two assertions.
\end{proof}

Here, the intersection poset is $N(\max(L))$, which we can simplify
further.  Regard $L$ as a poset under reverse inclusion.
Let us define an order-preserving map $\phi\colon N(\Cover_L)\to L$ by
\begin{equation}
\label{eq:phi gmac}
\phi((X,Y)^{\tau_1},\ldots,(X,Y)^{\tau_k})=\tau_1\cap\cdots\cap\tau_k,
\end{equation}
and a rank function $\rho\colon L\to\Z$ by 
$\rho(\sigma)=-\abs{\sigma}$. 

\begin{proposition}
\label{prop:gmac cover}
If $Y\subsetneq X$, then $(\Cover_L,\phi)$ is a strong combinatorial
cover of $\ZZ_L(X,Y)$.  
\end{proposition}
\begin{proof}
It is straightforward to check that $(\Cover_L,\phi)$ is a combinatorial
cover.  By Lemma~\ref{lem:gmac cover} and the Nerve Lemma, 
the map $\phi$ is a homotopy equivalence, which implies it is strong as well.
\end{proof}

\subsection{Toric complexes}
\label{subsec:tc}

Let $S^1=e^0\cup e^1$ be the circle, endowed with the standard 
cell decomposition.   The resulting moment-angle complex, 
$T_L:=\ZZ_L(S^1, e^0)$, 
is a subcomplex of the $n$-torus $T^n$, where $n=\abs{\sV}$.  

The fundamental group $G_L=\pi_1(T_L)$ is the 
{\em right-angled Artin group}\/ determined by the graph 
$\Gamma=L^{(1)}$, with presentation consisting of a generator 
$v$ for each vertex $v$ in $\sV$, and a commutator relation 
$vw=wv$ for each edge $\{v,w\}$ in $\Gamma$.
A classifying space for the group $G_L$ is the toric complex 
$T_{\Delta}$, where $\Delta=\Delta_{L}$ is the flag complex 
of $L$, i.e., the maximal simplicial complex with 
$1$-skeleton equal to the graph $\Gamma$. 

On the other hand, the cohomology ring $H^{\hdot}(T_L,\k)$ is the 
exterior Stanley--Reisner ring $\k\langle L\rangle$, with generators 
the duals $v^*$, and relations the monomials corresponding 
to the missing faces of $L$. 

\subsection{Cohen--Macaulay complexes}
\label{subsec:cm}

Recall that an $n$-dimensional simplicial complex $L$ is 
{\em Cohen--Macaulay}\/  if for each simplex $\sigma\in L$, 
the reduced cohomology $\widetilde{H}^*(\lk(\sigma),\Z)$ 
is concentrated in degree $n -\abs{\sigma}$ and is torsion-free. 
A similar definition holds over a coefficient field $\k$. 

For a fixed coefficient ring, the Cohen--Macaulayness of 
$L$ is a topological property: it depends only on the homeomorphism 
type of $L$.  For $\sigma=\emptyset$, the condition means that 
$\widetilde{H}^*(L,\Z)$ is concentrated in degree $n$; it also implies 
that $L$ is pure, i.e., all its maximal simplices have dimension $n$. 

\begin{remark}
\label{rem:torus}
Consider the extreme case in which $L$ is a simplex,  
so that $T_L$ is a torus and $G=\Z^n$ for some $n\ge 1$.  
Then we return to the situation of Lemma~\ref{lem:torus}, and we
have $H^p(G,A)=0$ 
for $p\neq n$ if and only if either $A$ is a maximal
Cohen--Macaulay module over $\k[G]_\one$, or its localization is zero.
\end{remark}

\begin{theorem}
\label{thm:cm toric}
Let $L$ be a $d$-dimensional Cohen--Macaulay complex over $\k$. 
Suppose $A$ is a $\MCM$-module over $\k[G_{\tau}]$, for each $\tau\in L$.  
Then $H^p(T_L,A)=0$, for all $p\neq d+1$.  If, moreover, $A$ is a free
$\k$-module, then so is $H^{d+1}(T_L,A)$.
\end{theorem}

\begin{proof}
Replace $T_L$ by a homotopy equivalent complex $\ZZ_L(S^1,Y)$, where
$Y$ is a contractible open neighbourhood of $e^0$.  Then it is enough 
to show the corresponding vanishing property for $H^p(\ZZ_L(S^1,Y),\F)$,
where $\F$ is the locally constant sheaf on $\ZZ_L(S^1,Y)$ defined by $A$. 

The space $\ZZ_L(S^1,Y)$ has an
open, strong combinatorial cover from Proposition~\ref{prop:gmac cover}.  
The spectral sequence of Corollary \ref{cor:Davis_trick} has $E_2$ term 
\begin{equation}
\label{eq:e2again}
E_2^{pq}=\bigoplus_{\tau\in L}  
H^{p+\abs{\tau}}_c(\st_L(\tau)\setminus \lk_L(\tau); 
H^{q-\abs{\tau}}((S^1,Y)^\tau,\Fr_\tau)).
\end{equation}

Since the complement $\st_L(\tau)\setminus \lk_L(\tau)$ is contractible, 
the coefficient local system is necessarily trivial.  Moreover,  
$(S^1,Y)^\tau\simeq T^{\abs{\tau}}$, a torus, so by 
Remark~\ref{rem:torus}, 
\begin{equation}
\label{eq:hqs1}
H^{q}(T^{\abs{\tau}},\Fr_\tau)=
H^q\big(T^{\abs{\tau}},\res^{G_L}_{G_\tau}A\big)=
\begin{cases}
A_{G_\tau} & \text{if $q=\abs{\tau}$;}\\[2pt]
0 & \text{otherwise.}
\end{cases}
\end{equation}
It follows that
\begin{equation}
\label{eq:e2pqbig}
E_2^{pq}= 
\bigoplus_{\substack{\tau\in L: \, q=2\abs{\tau}}}
H^{p+\abs{\tau}}_c(\st_L(\tau)\setminus \lk_L(\tau);A_{G_\tau}).
\end{equation}

By our assumption that $L$ is a $d$-dimensional Cohen--Macaulay 
complex over $\k$, we know that 
\begin{equation}
\label{eq:tildeh}
\widetilde{H}^{p+\abs{\tau}-1}(\lk_L(\tau),A_{G_\tau})=0,
\end{equation}
except possibly for $p+\abs{\tau}-1=d-\abs{\tau}$.  
In that case, $E_2^{pq}=0$, except possibly for $p+q=d+1$.
Thus, the spectral sequence collapses at the $E_2$ page, 
and the claim follows from Corollary~\ref{cor:Davis_trick}.

If, moreover, $A$ is free over $\k$, so is any module of coinvariants.
If $L$ is Cohen--Macaulay, then the nonzero summands of \eqref{eq:e2pqbig}
are free as well, and it follows that $H^{d+1}(T_L,A)$ is too.
\end{proof}
In a special case, we obtain a converse, along the lines of 
\cite[Theorem~C]{BM01}:

\begin{theorem}
\label{thm:cmk}
If $\k=\Z$ or a field, and $A=\k[G_L^{\ab}]$, then $L$ is a $d$-dimensional
Cohen--Macaulay complex over $\k$ if and only if $H^p(T_L,A)=0$ for all
$p\neq d+1$ and $H^{d+1}(T_L,A)$ is a free $\k$-module. 
\end{theorem}

\begin{proof}
To check the implication ``$\Rightarrow$'', we note that 
$A$ is a $\MCM$-module, by Proposition~\ref{prop:MCMcriteria}, 
and a free $\k$-module.  The result follows
from Theorem~\ref{thm:cm toric}.

To check the converse, suppose $L$ has dimension $d$, cohomology vanishes,
but $L$ is not Cohen--Macaulay.  First suppose $\k$ is a field.
Since $A$ is a ring, by naturality of restriction, $H^*((S^1,Y)^\tau,\F|_\tau)$
is an $A$-module for all $\tau$.  Then, as above, 
\begin{equation}
\label{eq:nonCM}
\widetilde{H}^{p+\abs{\tau}-1}(\lk_L(\tau),A_{G_\tau})\neq0
\end{equation}
for some $p<d-2\abs{\tau}+1$, so $E^{pq}_2\neq0$ where $q=2\abs{\tau}$.
Our hypothesis implies $E_\infty^{pq}=0$ unless $p+q=d+1$, so for some
(least) integer $r\geq2$, the differential $d_r^{pq}$ is nonzero.
By construction, summands of the
target of $d_r^{pq}$ are indexed by those simplices 
$\sigma$ for which $\sigma\subsetneq\tau$,
so by restricting and projecting, there exists a simplex 
$\sigma$ and a nonzero map
\begin{equation}
\label{eq:nonzero d}
d(\tau,\sigma)\colon 
\widetilde{H}^{p+\abs{\tau}-1}(\lk_L(\tau),A_{G_\tau})\to
\widetilde{H}^{p+\abs{\sigma}-1}(\lk_L(\sigma),A_{G_\sigma}).
\end{equation}

We note that $A_{G_\tau}=A/I_\tau$, where $I_\tau$ is prime ideal of $A$.
so $d$ is a homomorphism of modules
over the domain $A_{G_\sigma}$. 
Since $\k$ is a field, $\widetilde{H}^{p+\abs{\sigma}-1}(\lk_L(\sigma),
A_{G_\sigma})$ is a free module.  The ideal 
$I_\tau$ annihilates \eqref{eq:nonCM}, so it annihilates the image of
$d(\tau,\sigma)$.  Since $I_\sigma\subsetneq I_\tau$, 
the image of $I_\tau$ is nonzero
in $A_{G_\sigma}$.  It follows that $d(\tau,\sigma)$ is the zero map,
a contradiction.

Now suppose $\k=\Z$.  By the Universal Coefficients Theorem and the previous
argument, $L$ is Cohen--Macaulay over all fields.  Then $L$ is Cohen--Macaulay
over $\Z$, again by Universal Coefficients.
\end{proof}

\newcommand{\arxiv}[1]
{\texttt{\href{http://arxiv.org/abs/#1}{arXiv:#1}}}
\newcommand{\arxi}[1]
{\texttt{\href{http://arxiv.org/abs/#1}{arxiv:}}
\texttt{\href{http://arxiv.org/abs/#1}{#1}}}
\newcommand{\doi}[1]
{\texttt{\href{http://dx.doi.org/#1}{doi:#1}}}
\renewcommand{\MR}[1]
{\href{http://www.ams.org/mathscinet-getitem?mr=#1}{MR#1}}
\newcommand{\MRh}[2]
{\href{http://www.ams.org/mathscinet-getitem?mr=#1}{MR#1 (#2)}}


\begin{thebibliography}{XXXX00}

\bibitem[BBCG10]{BBCG10}  Anthony Bahri, Martin Bendersky, 
Frederick R.~Cohen, Samuel Gitler, 
{\em The polyhedral product functor: a method of 
computation for moment-angle complexes, 
arrangements and related spaces}, Advances in Math. 
\textbf{225} (2010), no.~3, 1634--1668.
\MRh{2673742}{2012b:13053}

\bibitem[Be94]{Be94} Roman Bezrukavnikov,
{\em Koszul {DG}-algebras arising from configuration spaces},
Geom. Funct. Anal. \textbf{4} (1994), no.~2, 119--135.
\MRh{1262702}{95c:55011}

\bibitem[Bi13]{Bi13}   Christin Bibby, 
{\em Cohomology of abelian arrangements}, 
\arxiv{1310.4866v3}.

\bibitem[Bo84]{Bo84} Armand Borel et al.,
{\em Intersection cohomology}, Progress in Mathematics, vol.~50, 
Birkh\"auser, Boston, 1984.
\MRh{0788171}{88d:32024}

\bibitem[BT82]{BT}  Raoul Bott and Loring Tu, 
{\em Differential forms in algebraic topology},  
Graduate Texts in Mathematics, vol.~82, 
Springer-Verlag, New York-Berlin, 1982. 
\MRh{0658304}{83i:57016}  

\bibitem[BM01]{BM01} Noel Brady and John Meier, 
{\em Connectivity at infinity for right angled Artin
groups}, Trans. Amer. Math. Soc. \textbf{353} (2001), 
no.~1, 117--132.
\MRh{1675166}{2001b:20068}

\bibitem[DCP95]{dCP95} Corrado~De~Concini and Claudio~Procesi, 
\emph{Wonderful models of subspace arrangements}, 
Selecta Math. (N.S.) \textbf{1} (1995), no.~3, 459--494.
\MRh{1366622}{97k:14013}

\bibitem[Cr67]{Cr67} Henry Crapo, 
{\em A higher invariant for matroids}, 
J. Combinatorial Theory \textbf{2} (1967), 406--417. 
\MRh{0215744}{35 \#6579}

\bibitem[Da12]{Da12} Michael W. Davis, 
{\em Right-angularity, flag complexes, asphericity}, 
Geom. Dedicata \textbf{159} (2012), no.~1, 239--262.
\MR{2944529}

\bibitem[DJLO11]{DJLO11}  Michael W. Davis, Ian Leary, 
Tadeusz Januszkiewicz, and Boris Okun, 
\emph{Cohomology of hyperplane complements with group 
ring coefficients}, Intern. Math. Res. Notices \textbf{2011}, no.~9, 
2110--2116.
\MR{2806559}  

\bibitem[DO12]{DO12} Michael W. Davis and Boris Okun,
{\em Cohomology computations for Artin groups, 
Bestvina-Brady groups, and graph products}, 
Groups Geom. Dyn.  \textbf{6} (2012), no.~3, 485--531.
\MR{2961283}

\bibitem[DS13]{DS13} Michael W. Davis and Simona Settepanella, 
{\em 	Vanishing results for the cohomology of complex 
toric hyperplane complements},  Publ. Mat. \textbf{57} (2013), 
no. 2, 379--392.
\MR{3114774}

\bibitem[Deh62]{deH62} Ren{\'e} Deheuvels, 
\emph{Homologie des ensembles ordonn\'es et des espaces topologiques}, 
Bull. Soc. Math. France \textbf{90} (1962), 261--321.
\MR{0168614 (29 \#5874)}

\bibitem[De14]{De14} Graham Denham, 
{\em Toric and tropical compactifications of hyperplane complements}, 
Ann. Fac. Sci. Toulouse Math.  \textbf{23} (2014), no.~2, 297--333. 
\MR{3205595}

\bibitem[DS07]{DS07} Graham Denham and Alexander I. Suciu, 
{\em Moment-angle complexes, monomial ideals, and Massey products}, 
Pure Appl. Math. Q. \textbf{3} (2007), no.~1, 25--60. 
\MRh{2330154}{2008g:55028}

\bibitem[DS14]{DS14} Graham Denham and Alexander I. Suciu, 
{\em Multinets, parallel connections, and {M}ilnor fibrations 
of arrangements},  Proc. London Math. Soc. \textbf{108} (2014), 
no.~6, 1435--1470.    
\MR{3218315}

\bibitem[DSY$_1$]{DSY-duality}  Graham Denham, Alexander I. Suciu, 
and Sergey Yuzvinsky, {\em Abelian duality and propagation 
of resonance}, in preparation. 

\bibitem[DSY$_2$]{DSY-novikov}  Graham Denham, Alexander I. Suciu, 
and Sergey Yuzvinsky, {\em Vanishing Novikov homology}, in preparation.  

\bibitem[Di04]{Di04} Alexandru Dimca, 
{\em Sheaves in topology}, Universitext, Springer-Verlag, Berlin, 2004.
\MRh{2050072}{2005j:55002}

\bibitem[DPS09]{DPS} Alexandru Dimca, Stefan Papadima, 
and Alexander I. Suciu,
\emph{Topology and geometry of cohomology jump loci}, 
Duke Math. Journal \textbf{148} (2009), no.~3, 405--457.
\MRh{2527322}{2011b:14047}

\bibitem[EPY03]{EPY03}
David Eisenbud, Sorin Popescu, and Sergey Yuzvinsky, 
\emph{Hyperplane  arrangement cohomology and monomials 
in the exterior algebra}, Trans. Amer. Math. Soc. \textbf{355} 
(2003), no.~11, 4365--4383.
\MRh{1986506}{2004g:52036}

\bibitem[ESV92]{ESV92} H\'{e}l\`{e}ne Esnault, Vadim Schechtman, 
and Eckart Viehweg,
{\em Cohomology of local systems on the complement of 
hyperplanes}, Invent. Math. \textbf{109} (1992), 557-561.  
\MRh{1176205}{93g:32051}  Erratum, ibid. \textbf{112} (1993), 
no.~2, 447. \MRh{1213111}{94b:32061}

\bibitem[FY04]{FY04} Eva~Maria Feichtner and Sergey Yuzvinsky, 
\emph{Chow rings of toric varieties defined by atomic lattices}, 
Invent. Math. \textbf{155} (2004), no.~3, 515--536. 
\MRh{2038195}{2004k:14009}

\bibitem[Fo66]{Fo66} Jon Folkman, 
{\em The homology groups of a lattice}, 
J. Math. Mech. \textbf{15} (1966), 631--636. 
\MRh{0188116}{32 \#5557}  

\bibitem[FG02]{FG02} Klaus Fritzsche and Hans Grauert, 
{\em From holomorphic functions to complex manifolds},  
Grad.Texts in Math., vol. 213. Springer-Verlag, New York, 2002.
\MRh{1893803}{2003g:32001}

\bibitem[God58]{God} Roger Godement, 
\emph{Topologie alg\'ebrique et th\'eorie des faisceaux},
Actualit\'es Sci. Ind. no.~1252, Publ. Math. Univ. Strasbourg, 
no.~13, Hermann, Paris, 1958. 
\MRh{0102797}{21 \#1583}

\bibitem[Ha86]{Ha86} John Harer, 
{\em The virtual cohomological dimension of the mapping class group 
of an orientable surface}, Invent. Math. \textbf{84} (1986), no.~1, 157--176. 
\MRh{0830043}{87c:32030}

\bibitem[Ive86]{Iv} Birger Iversen, 
\emph{Cohomology of sheaves}, Universitext, Springer-Verlag,
Berlin, 1986. 
\MRh{842190}{87m:14013}

\bibitem[KS90]{KSbook} Masaki Kashiwara and Pierre Schapira, 
\emph{Sheaves on manifolds}, Grundlehren Math. Wiss., vol. 292, 
Springer-Verlag, Berlin, 1990. 
\MRh{1074006}{92a:58132}

\bibitem[JM05]{JM05} Craig Jensen and John Meier, 
{\em The cohomology of right-angled Artin groups with group 
ring coefficients}, Bull. London Math. Soc. \textbf{37} (2005),
no.~5, 711--718.
\MRh{2164833}{2006m:20054}

\bibitem[Ko86]{Ko86} Toshitake Kohno,
{\em Homology of a local system on the complement of hyperplanes},  
Proc. Japan Acad. Ser. A Math. Sci. \textbf{62} (1986), no.~4, 144--147.
\MRh{0846350}{87i:32019}

\bibitem[LV12]{LV12} Andrey Levin and Alexander Varchenko,
{\em Cohomology of the complement to an elliptic arrangement},  
in:  {\em Configuration spaces: geometry, combinatorics and topology}, 
373--388, CRM Series, vol.~14, Ed. Norm., Pisa, 2012. 
\MR{3203648}

\bibitem[PS09]{PS09} Stefan Papadima and Alexander I. Suciu,
{\em Toric complexes and {A}rtin kernels}, Adv. Math. 
\textbf{220} (2009), no.~2, 441--477.
\MRh{2466422}{2010h:57007}

\bibitem[Qu78]{Qu78}  Daniel Quillen, 
{\em Homotopy properties of the poset of nontrivial $p$-subgroups 
of a group}, Adv. in Math. \textbf{28} (1978), no. 2, 101--128. 
\MRh{0493916}{80k:20049}  

\bibitem[STV95]{STV95} Vadim Schechtman, Hiroaki Terao, 
and Alexander Varchenko, 
{\em Local systems over complements of hyperplanes 
and the {K}ac-{K}azhdan condition for singular vectors}, 
J. Pure Appl. Algebra \textbf{100} (1995), no.~1-3, 93--102.
\MRh{1344845}{96j:32047}

\bibitem[Sc03]{Sch03} J{\"o}rg Sch{\"u}rmann, 
\emph{Topology of singular spaces and constructible sheaves}, 
Mathematics Institute of the Polish Academy of Sciences, 
Mathematical Monographs, vol.~63, Birkh\"auser Verlag, 
Basel, 2003. 
\MRh{2031639}{2005f:32053}

\bibitem[Sma57]{Sma57} Stephen Smale, 
\emph{A {V}ietoris mapping theorem for homotopy}, Proc. Amer.
Math. Soc. \textbf{8} (1957), 604--610. 
\MRh{0087106}{19,302f}

\end{thebibliography}
\end{document}